\documentclass[12pt]{amsart}
\usepackage{blindtext}
\usepackage{url}
\usepackage{tabularx}
\usepackage[utf8]{inputenc}
\usepackage[margin=2.5cm]{geometry}
\usepackage[utf8]{inputenc}
\usepackage{amsmath}
\usepackage{amsfonts}
\usepackage{color}
\usepackage{slashbox}
\usepackage{amssymb}
\usepackage{amsthm}
\usepackage{array}
\usepackage{rotating}
\usepackage{physics}
\usepackage{tikz}
\usepackage{braids}
\usepackage{tikz-cd}
\usepackage[bottom]{footmisc}
\tikzset{%
	symbol/.style={%
		draw=none,
		every to/.append style={%
			edge node={node [sloped, allow upside down, auto=false]{$#1$}}}
	}
}
\usepackage[all]{xy}
\usepackage{cite}
\usepackage{mathdots}
\usepackage{pstricks}
\usepackage{mathrsfs}
\usepackage{setspace}
\usepackage{hyperref}
\usepackage{fancyhdr}
\theoremstyle{theorem}
\newtheorem{theorem}{Theorem}
\numberwithin{theorem}{subsection}
\newtheorem{lemma}[theorem]{Lemma}
\newtheorem{proposition}[theorem]{Proposition}
\newtheorem{corollary}[theorem]{Corollary}
\theoremstyle{conjecture}

\theoremstyle{definition}
\newtheorem{definition}[theorem]{Definition}
\newtheorem{example}[theorem]{Example}
\theoremstyle{remark}
\newtheorem{remark}[theorem]{Remark}
\theoremstyle{diagram}

\theoremstyle{notation}
\newtheorem{notation}[theorem]{Notation}
\title{Enriched Kleisli objects for pseudomonads}
\author{Adrian Miranda}
\thanks{This material is based upon work done for my PhD thesis. I am grateful for having been supported by the MQRES PhD Scholarship, 20192497, and by EPSRC under grant EP/V002325/2. I thank Steve Lack for his guidance while I was conducting this research, and Nicola Gambino for his guidance while I was preparing this paper. 2020 Mathematics subject classification 18N15. Keywords: pseudomonads, Kleisli constructions, enriched category theory, $2$-categories, presentations, projective cofibrancy.}
\address{Department of Mathematics, University of Manchester, 
	United Kingdom}
\email{adrian.miranda@manchester.ac.uk}
\begin{document}
	
	\maketitle
	
	\begin{abstract}
		\noindent A pseudomonad on a $2$-category whose underlying endomorphism is a $2$-functor can be seen as a diagram $\mathbf{Psmnd} \rightarrow \mathbf{Gray}$ for which weighted limits and colimits can be considered. The $2$-category of pseudoalgebras, pseudomorphisms and $2$-cells is such a $\mathbf{Gray}$-enriched weighted limit \cite{Coherent Approach to Pseudomonads}, however neither the Kleisli bicategory nor the $2$-category of free pseudoalgebras are the analogous weighted colimit \cite{Formal Theory of Pseudomonads}. In this paper we describe the actual weighted colimit via a presentation, and show that the comparison $2$-functor induced by any other pseudoadjunction splitting the original pseudomonad is bi-fully faithful. As a consequence, we see that biessential surjectivity on objects characterises left pseudoadjoints whose codomains have an `up to biequivalence' version of the universal property for Kleisli objects. This motivates a homotopical study of Kleisli objects for pseudomonads, and to this end we show that the weight for Kleisli objects is cofibrant in the projective model structure on $[\mathbf{Psmnd}^\text{op}, \mathbf{Gray}]$.
	\end{abstract}
	
	\onehalfspacing

\tableofcontents

\section{Introduction}\label{Introduction}

\subsection{Background and Motivation}

\noindent Many structures in mathematics can be expressed as algebras for suitable monads $\left(S, \eta, \mu\right)$ on suitable base categories $A$. Such algebras can be described via presentations, which view them as quotients of the easier to understand free algebras $\left(SX, \mu_{X}\right)$ \cite{Beck's Theorem For Pseudomonads}. The category of free algebras is also equivalent to a category whose objects are those of the base $A$, whose morphisms are of the form $f\in A\left(X, SY\right)$, and whose identities and composition are given in terms of $\eta$ and $\mu$ \cite{Kleisli standard construction induced by a pair of adjoint functors}. In this way, monads model not just algebras but also `maps with generalised outputs'. The category of algebras, or Eilenberg-Moore category \cite{Eilenberg Moore}, is a $\mathbf{Cat}$ enriched weighted limit of the data in the original monad, while the category of maps with generalised outputs, or Kleisli category, is the analogous weighted colimit \cite{Formal Theory of Monads}.
\\
\\
\noindent Pseudomonads generalise monads by allowing the usual monad axioms to hold up to invertible modifications which are subject to further axioms (Section 8 of \cite{Doctrines fully faithful adjoint string}). There are several Kleisli like constructions for pseudomonads that have been considered in the literature. Of these, the constructions of greatest importance seem to be the following.

\begin{itemize}
	\item The $2$-category of free pseudoalgebras. Once again, general pseudoalgebras can be expressed as certain colimits of free ones \cite{Beck's Theorem For Pseudomonads}.
	\item The bicategory of Kleisli arrows $f: X \rightarrow SY$. This construction has received much attention in the literature due to its myriad applications such as to bisimulation \cite{bisimulation}, topos theory \cite{Niefield Exponentiability}, universal algebra \cite{Hyland elements of a theory of algebraic theories}, and combinatorial species \cite{Fiore Gambina Hyland Winskel generalised species} \cite{Gambino Garner Vasilakopoulou}. We recall it in Subsection \ref{section Kleisli Bicategory}.
\end{itemize}
\noindent However, these both fail to have a purely $\mathbf{Gray}$-enriched universal property. The $\mathbf{Gray}$-enriched notion has been considered in Remark 4.7 and the discussion preceding it in \cite{Formal Theory of Pseudomonads}, and it was also mentioned in the discussion following Theorem 5.3 of \cite{Coherent Approach to Pseudomonads}. Since $\mathbf{Gray}$ is cocomplete when considered as category enriched over itself (3.74 of \cite{Kelly Basic Concepts of Enriched Category Theory}), a $2$-category having the enriched universal property must exist. Section \ref{Section definition and presentation of Gray enriched Kleisli objects for pseudomonads} recalls the weight defining the $\mathbf{Gray}$-enriched Kleisli construction and the corresponding presentation for the actual colimit. This presentation is a categorification of a description of the Kleisli category of a monad which may seem unfamiliar even in the one-dimensional setting, and which we describe in Subsection \ref{Subsection Kleisli via formal algebra structures}.
\\
\\
\noindent Since $2$-categories of free pseudoalgebras, and Kleisli bicategories, seem more useful than $\mathbf{Gray}$-enriched Kleisli objects which have a simpler universal property, an understanding of the relationship between the various constructions is of interest. As Theorem \ref{comparison from Gray Kleisli is bi-fully-faithful} establishes, the comparison from the $\mathbf{Gray}$-enriched Kleisli object to any other pseudoadjunction splitting is bi-fully-faithful. Hence left pseudoadjoints whose targets have an `up to biequivalence' version of the universal property are characterised by the property of being biessentially surjective on objects. The property of biessential surjectivity on objects is closed under composition, and closure under composition is a key ingredient in the development of free cocompletions under Kleisli objects for ordinary monads \cite{FTM2}. This motivates homotopical perspectives on Kleisli objects for pseudomonads, and Section \ref{Section Homotopy Theoretic Aspects} is devoted to establishing useful homotopical properties. These directions will be pursued further after the theory of tricategorical limits and colimits is developed in a forthcoming paper \cite{Miranda Tricategorical Limits and Colimits}. The aim of the present paper is to collect certain useful facts about the weight and about the $\mathbf{Gray}$-enriched Kleisli object itself.

\subsection{Outline of the paper}

\noindent The main results of this paper are Theorem \ref{when is comparison from gray kleisli a biequivalence} and Corollary \ref{TriKL weight is cofibrant}. These follow easily from the results preceding them; Theorem \ref{comparison from Gray Kleisli is bi-fully-faithful} and Proposition \ref{Pseudoalgebras via cofibrant weighted limits} respectively. An outline of the content of each Section of this paper is given below, and a summary of the key ideas and techniques used in this paper follows in Subsection \ref{Subsection Key ideas and techniques}. 

\begin{itemize}
	\item Subsection \ref{Subsection free pseudoalgebras} reviews $2$-categories of pseudoalgebras for pseudomonads and Subsection \ref{section Kleisli Bicategory} reviews Kleisli bicategories.
	\item Section \ref{Section weights and presentations} derives presentations for Kleisli objects from their universal properties as weighted colimits. Subsection \ref{Subsection Kleisli via formal algebra structures} describes the Kleisli category of a monad via the presentation inherited from its universal property. Subsection \ref{Section definition and presentation of Gray enriched Kleisli objects for pseudomonads} recalls the definitions of pseudomonads and pseudoadjunctions, and the weight for Kleisli objects. Subsection \ref{Subsection the presentation of the Kleisli object of a pseudomonad} gives a presentation for the Kleisli object $\mathcal{A}_{S}$ of a pseudomonad $\left(\mathcal{A}, S\right)$. In particular, the universal cocone is described in Remark \ref{Gray Kleilsi's universal cocone as a weighted colimit}. This is simplified by the coherence results for pseudomonads established in \cite{Coherent Approach to Pseudomonads}.
	\item Section \ref{Section the canonical comparison induced by a pseudoadjunction splitting the pseudomonad} studies canonical comparison $2$-functors induced out of Kleisli objects $\mathcal{A}_{S}$ by pseudoadjunctions splitting $\left(\mathcal{A}, S\right)$. Subsection \ref{Subsection comparison to pseudoalgebras} describes the canonical comparison $2$-functor from the Kleisli object $\mathcal{A}_{S}$ to the Eilenberg-Moore object. An explicit biequivalence inverse to a factorisation of the comparison is given in Appendix \ref{Appendix explicit inverse to comparison}. Subsection \ref{Subsection bi-fully-faithfulness} establishes bi-fully faithfulness (Theorem \ref{comparison from Gray Kleisli is bi-fully-faithful}), and characterises when such comparisons are weak equivalences of $2$-categories in the sense of \cite{Quillen 2-cat}, here called \emph{weak biequivalences} (Definition \ref{definition weak biequivalence}, Theorem \ref{when is comparison from gray kleisli a biequivalence}). Specifically, this occurs if and only if the left pseudoadjoint is biessentially surjective on objects. Subsection \ref{Subsection Applications} applies these results to simplify the description of $2$-cells in $\mathcal{A}_{S}$ (Corollary \ref{2-cells in A_S are 2-cells of free pseudoalgebras}), and to show that left pseudoadjoints whose codomains satisfy the $\mathbf{Gray}$-enriched universal property are not closed under composition (Corollary \ref{Gray Kleisli not closed under composition}).
	\item Section \ref{Section Homotopy Theoretic Aspects} works towards an understanding of Kleisli objects for pseudomonads which only satisfy a weaker universal property, using the tools of enrichment over monoidal model categories. This is motivated by examples such as Kleisli bicategories and $2$-categories of free pseudoalgebras, as well as by the better properties characterising left adjoints whose codomains satisfy these weaker universal properties. Subsection \ref{Subsection Homotopical preliminaries} reviews relevant homotopical notions, including Lack's model structure on $2$-$\mathbf{Cat}$ and projective model structures. Subsection \ref{Subsection Kleilsi weight is projective cofibrant} records that $\mathbf{Gray}$-enriched Kleisli left adjoints are cofibrations in the Lack model structure (Proposition \ref{gray kleisli left adjoint a cofibration}), and establishes projective cofibrancy of the weight for Kleisli objects (Corollary \ref{TriKL weight is cofibrant}). This will be used in the forthcoming \cite{Miranda Tricategorical Limits and Colimits} to clarify the relationship between the enriched universal property of $\mathcal{A}_{S}$, and the weaker universal properties of more familiar Kleisli constructions for pseudomonads.
\end{itemize} 

\begin{notation}\label{notation morphisms in Kleisli bicategories}
	When a morphism $f: X \rightarrow SY$ in $\mathcal{A}$ is being considered as a morphism from $X$ to $Y$ in the Kleisli bicategory, it will be denoted as $f: X \nrightarrow Y$.
\end{notation}

\subsection{Key ideas and techniques}\label{Subsection Key ideas and techniques}

\subsubsection{Enriched category theory}

\noindent The most general notion of a pseudomonad that one might consider in two-dimensional category theory would have a bicategory $\mathcal{A}$ as its base, and pseudofunctor $S: \mathcal{A} \rightsquigarrow \mathcal{A}$ as its endomorphism; a pseudomonad $\mathbb{S}$ in the tricategory $\mathbf{Bicat}$ \cite{Gurski Coherence in Three Dimensional Category Theory}. However, two-dimensional strictification $\mathbf{st}_{2}: \mathbf{Bicat} \rightsquigarrow \mathbf{Gray}$ \cite{Campbell Strictification} enables the base to be replaced with a biequivalent $2$-category $\mathbf{st}_{2}\left(\mathcal{A}\right)$ and the endomorphism to be replaced with a $2$-functor in a compatible way. Specifically, one uses cofibrancy of $\mathbf{Psmnd}$ to replace the trihomomorphism \begin{tikzcd}
	\mathbf{Psmnd} \arrow[r, "\mathbb{S}"] & \mathbf{Bicat} \arrow[r, "\mathbf{st}_{2}"] & \mathbf{Gray}
\end{tikzcd} with a trinaturally equivalent $\mathbf{Gray}$-functor (Proposition 4.1.1, \cite{Miranda strictifying operational coherences}).
\\
\\
\noindent As justified by the previous paragraph, we restrict our attention to pseudomonads whose bases are $2$-categories and whose underlying endomorphisms are $2$-functors. Such simplifications do not lose any generality up to biequivalence, and they allow us to apply the techniques of enriched category theory \cite{Kelly Basic Concepts of Enriched Category Theory} over the base symmetric monoidal closed category $\left(2\text{-}\mathbf{Cat}, \otimes, \mathbf{1}, \mathbf{Gray}\left(-, ?\right)\right)$. Here $\otimes$ denotes the $\mathbf{Gray}$-tensor product \cite{Gray Formal Category Theory} and $\mathbf{Gray}\left(\mathcal{A}, \mathcal{B}\right)$ is the $2$-category whose objects are $2$-functors from $\mathcal{A}$ to $\mathcal{B}$, morphisms are pseudonatural transformations, and $2$-cells are modifications. When being considered as a base for enrichment, or as a category enriched over itself, this structure is denoted $\mathbf{Gray}$.

\subsubsection{Presentations for weighted colimits via generators and relations}

\noindent The base $\mathbf{Gray}$ is a locally finitely presentable symmetric monoidal closed category \cite{Adamek Rosicky Locally Presentable and Accessible Categories}. As such, any weighted colimit can be described via a presentation involving generators and relations. In general, such presentations for colimits can be constructed using data that must appear in cocones. Subsection \ref{Subsection the presentation of the Kleisli object of a pseudomonad} exploits the coherence theorem for pseudomonoids and pseudoalgebras of \cite{Coherent Approach to Pseudomonads}, to describe such a presentation when the weight is that for Kleisli objects of pseudomonads. In particular, this allows a presentation involving only finitely many different types of generators and relations.

\subsubsection{Homotopy theory of $2$-categories}

\noindent As well as being locally presentable, $\mathbf{Gray}$ is also a monoidal model category when equipped with the Lack model structure \cite{Quillen 2-cat}, in which the weak equivalences are the weak biequivalences, recalled in Definition \ref{definition weak biequivalence}. As shown in Example 3.1 of \cite{Quillen 2-cat}, the biequivalence inverse of a weak biequivalence is typically a pseudofunctor rather than a morphism of $\mathbf{Gray}$. The main result of Section \ref{Section the canonical comparison induced by a pseudoadjunction splitting the pseudomonad} is a characterisation of when the comparison from the enriched Kleisli object is a weak biequivalence. This motivates a homotopical perspective on Kleisli objects for pseudomonads, which is developed in Section \ref{Section Homotopy Theoretic Aspects}.

\section{Related notions}\label{Section related notions and perspectives}

\noindent Although this paper will mostly focus on $\mathbf{Gray}$-enriched Kleisli objects for pseudomonads, there are more familiar constructions which have weaker universal properties. Subsection \ref{Subsection free pseudoalgebras} reviews the $2$-category of pseudoalgebras, including its universal property as an Eilenberg-Moore object for the pseudomonad. The full sub $2$-category on free pseudoalgebras categorifies one perspective on the ordinary Kleisli construction for monads on categories. Another perspective is categorified by the Kleisli bicategory, which we review in Subsection \ref{section Kleisli Bicategory}. As mentioned in the introduction, neither of these constructions has a $\mathbf{Gray}$-enriched universal property. Rather than $\mathbf{Gray}$-natural isomorphisms, their universal properties involve trinatural biequivalences. A study of tricategorical perspectives is deferred to the forthcoming \cite{Miranda Tricategorical Limits and Colimits}.

\subsection{Pseudoalgebras}\label{Subsection free pseudoalgebras}

\noindent We recall from section 9 of \cite{Doctrines fully faithful adjoint string} the $2$-category of pseudoalgebras $\mathcal{A}^S$ for a pseudomonad $\left(\mathcal{A}, S\right)$, and the Eilenberg Moore pseudoadjunction. The explicit descriptions of pseudoalgebras, pseudomorphisms and $2$-cells of pseudoalgebras will be helpful in understanding the constructions in Proposition \ref{Pseudoalgebras via cofibrant weighted limits}. The $2$-category structure is easily inferred from the description of these data.

\begin{definition}
\hspace{1mm}
	\begin{enumerate}
		\item A \emph{pseudoalgebra} $\left(X, \phi, u, m\right)$ consists of an object $X \in \mathcal{A}$, a morphism $\phi: SX \rightarrow X$, and two invertible $2$-cells $u: 1_{X} \Rightarrow \phi.\eta_{X}$ and $m: \phi.S\phi \Rightarrow \phi.\mu_{X}$ which satisfy the equations depicted below.
		
		$$\begin{tikzcd}[font=\fontsize{9}{6}]
			S^{3}X \arrow[dd, "\mu_{SX}"']
			\arrow[rr, "S^{2}\phi"]\arrow[rd, "S\mu_{X}"description]
			&{}\arrow[d, Rightarrow, shorten = 5, "Sm"]& S^{2}X\arrow[rd, "S\phi"]
			\\
			&S^{2}X\arrow[rr, "S\phi"]\arrow[dd, "\mu_{X}"]\arrow[d, Rightarrow, shorten = 5, shift right = 12, "\alpha_{X}"]
			&{}\arrow[dd, Rightarrow, shorten = 15, "m"]
			&SX\arrow[dd, "\phi"]
			\\
			S^{2}X\arrow[rd, "\mu_{X}"']&{}&&&=
			\\
			&SX \arrow[rr, "\phi"']
			&{}& X
		\end{tikzcd}\begin{tikzcd}[font=\fontsize{9}{6}]
			S^{3}X \arrow[dd, "\mu_{SX}"']
			\arrow[rr, "S^{2}\phi"]
			&{}\arrow[dd, Rightarrow, shorten = 20, Rightarrow, "\mu_{\phi}"]
			& S^{2}X\arrow[rd, "S\phi"]\arrow[dd, "\mu_{X}"description]
			\\
			&&&SX\arrow[dd, "\phi"]
			\arrow[d, Rightarrow, shorten = 5, shift right = 12, "m"]
			\\
			S^{2}X\arrow[rd, "\mu_{X}"']\arrow[rr, "S\phi"]
			&{}\arrow[d, Rightarrow,"m"]
			&SX\arrow[rd, "\phi"]
			&{}
			\\
			&SX \arrow[rr, "\phi"']
			&{}& X
		\end{tikzcd}$$
		
		$$\begin{tikzcd}[font=\fontsize{9}{6}]
			&&{}\arrow[dd, Rightarrow, shorten = 10, "Su"]
			&&SX\arrow[rrdd, "\phi"]\arrow[dddd, Rightarrow, shorten = 15, "m"]
			\\
			\\
			SX \arrow[rrrruu, bend left = 30, "1_{SX}"]
			\arrow[rrrrdd, bend right = 30, "1_{SX}"']
			\arrow[rr, "S\eta_{X}"]
			&& S^{2}X \arrow[rruu, "S\phi"]\arrow[rrdd, "\mu_{X}"]\arrow[dd, Rightarrow, shorten = 10, "\rho_{X}"]
			&&&&X&=&SX\arrow[rr, bend left = 20, "\phi"name=A]\arrow[rr, bend right = 20, "\phi"'name=B] && X
			\\
			\\
			&&{}&&SX\arrow[rruu, "\phi"]
			\arrow[from=A, to=B, Rightarrow, shorten = 5, "1_{\phi}"]
		\end{tikzcd}$$
		\item A \emph{pseudomorphism} $\left(f, \overline{f}\right): \left(X, \phi, u, m\right) \rightarrow \left(Y, \psi, n, v\right)$ consists of a morphism $f: X \rightarrow Y$ and an invertible $2$-cell $\overline{f}: \psi.Sf \Rightarrow f.\phi$ satisfying the axioms depicted below.

		$$\begin{tikzcd}[font=\fontsize{9}{6}]
			S^{2}X \arrow[dd, "\mu_{SX}"']
			\arrow[rr, "S^{2}f"]\arrow[rd, "S\phi"description]
			&{}\arrow[d, Rightarrow, shorten = 5, "S\overline{f}"]& S^{2}Y\arrow[rd, "S\psi"]
			\\
			&SX\arrow[rr, "Sf"]\arrow[dd, "\phi"]\arrow[d, Rightarrow, shorten = 5, shift right = 12, "m"]
			&{}\arrow[dd, Rightarrow, shorten = 15, "\overline{f}"]
			&SY\arrow[dd, "\psi"]
			\\
			SX\arrow[rd, "\phi"']&{}&&&=
			\\
			&X \arrow[rr, "f"']
			&{}& Y
		\end{tikzcd}\begin{tikzcd}[font=\fontsize{9}{6}]
			S^{2}X \arrow[dd, "\mu_{SX}"']
			\arrow[rr, "S^{2}f"]
			&{}\arrow[dd, Rightarrow, shorten = 15, "\mu_{f}"]& S^{2}Y\arrow[rd, "S\psi"]\arrow[dd, "\mu_{Y}"]
			\\
			&&{}
			&SY\arrow[dd, "\psi"]\arrow[d, Rightarrow, shift right = 10, "n"]
			\\
			SX\arrow[rd, "\phi"']\arrow[rr, "Sf"]&{}\arrow[d, Rightarrow, shorten = 5, "\overline{f}"]
			&SY\arrow[rd, "\psi"]&{}
			\\
			&X \arrow[rr, "f"']
			&{}& Y
		\end{tikzcd}$$
		
		$$\begin{tikzcd}[font=\fontsize{9}{6}]
			X \arrow[rr, "\eta_{X}"]\arrow[rrdd, bend right = 30, "1_{X}"'] &{}\arrow[dd, Rightarrow, shorten = 18, "u"]
			& SX\arrow[dd, "\phi"] \arrow[rr, "Sf"] &{}\arrow[dd, Rightarrow, shorten = 18, "\overline{f}"]
			& SY\arrow[dd, "\psi"]
			\\
			&&&&&=
			\\
			&{}& X \arrow[rr, "f"'] &{}
			& Y
		\end{tikzcd}\begin{tikzcd}[font=\fontsize{9}{6}]
			X \arrow[rr, "\eta_{X}"]\arrow[rrdd, "f"']
			&& SX \arrow[dd, Rightarrow,shorten = 15, "\eta_{f}"]
			\arrow[rr, "Sf"] && SY\arrow[dd, "\psi"]
			\\
			&&&{}\arrow[d, shorten = 5,Rightarrow, "v"]
			\\
			&&Y\arrow[rruu, "\eta_{Y}"]
			\arrow[rr, "1_{Y}"']&{}&Y
		\end{tikzcd}$$
		\item Let $\left(f, \overline{f}\right)$ be as above and let $\left(g, \overline{g}\right)$ be another parallel pseudomorphism. A $2$-cell of pseudoalgebras $\gamma: \left(f, \overline{f}\right) \Rightarrow \left(g, \overline{g}\right)$ consists of a $2$-cell $\gamma: f \rightarrow g$ satisfying the equation depicted below.
		
		$$\begin{tikzcd}[font=\fontsize{9}{6}]
			SX \arrow[dd, "\phi"']\arrow[rr, bend left = 30, "Sf"name=A]\arrow[rr, bend right = 30, "Sg"'name=B] && SY\arrow[dd, "\psi"]
			\\
			&&&=&{}
			\\
			X \arrow[rr, bend right = 30, "g"'name=D]&& Y
			\arrow[from=A, to=B, Rightarrow, shorten = 5, "S\gamma"]\arrow[from=B, to=D, Rightarrow, shorten = 15, "\overline{g}"]
		\end{tikzcd}\begin{tikzcd}[row sep = 27, font=\fontsize{9}{6}]
			SX \arrow[dd, "\phi"']\arrow[rr, bend left = 30, "Sf"name=A]&& SY\arrow[dd, "\psi"]
			\\
			\\
			X \arrow[rr, bend left = 30, "f"name=C]
			\arrow[rr, bend right = 30, "g"'name=D]&& Y
			\arrow[from=A, to=C, Rightarrow, shorten = 15, "\overline{f}"]\arrow[from=C, to=D, Rightarrow, shorten = 5, "\gamma"]
		\end{tikzcd}$$
		
	\end{enumerate}
	
\end{definition}

\noindent Pseudoalgebras, pseudomorphisms and $2$-cells of pseudoalgebras form the Eilenberg-Moore object of the pseudomonad in $\mathbf{Gray}$. We review the universal property in Proposition \ref{EM universal property}, to follow, and review how this notion can be seen as a weighted limit in Definition \ref{Gray Kleisli def as a left kan extension}. In contrast to Proposition \ref{EM universal property} and to the one-dimensional setting, neither free pseudoalgebras nor Kleisli bicategories described in Subsection \ref{section Kleisli Bicategory} have an enriched universal property.

\begin{proposition}\label{EM universal property}
	\hspace{1mm}
	\begin{enumerate}
		\item There is a $2$-category $\mathcal{A}^{S}$ whose objects are pseudoalgebras, whose morphisms are pseudomorphisms, and whose $2$-cells are $2$-cells of pseudoalgebras.
		\item There is a $2$-functor $U^S: \mathcal{A}^{S} \rightarrow \mathcal{A}$ which sends
		\begin{itemize}
			\item a pseudoalgebra $\left(X, \phi, u, m\right)$ to $X$,
			\item a pseudomorphism $\left(f, \overline{f}\right)$ to $f$,
			\item a $2$-cell of pseudoalgebras $\gamma$, to itself viewed as a $2$-cell of $\mathcal{A}$.
		\end{itemize}
		\item There is a $2$-functor $F^{S}: \mathcal{A} \rightarrow \mathcal{A}^{S}$ which sends 
		\begin{itemize}
			\item an object $X$ to the pseudoalgebra $\left(SX, \mu_{X}, \rho_{X}, \alpha_{X}\right)$,
			\item a morphism $f: X \rightarrow Y$ to the pseudomorphism $\left(Sf, \mu_{f}\right)$,
			\item a $2$-cell $\gamma: f \Rightarrow g$ to the $2$-cell of pseudoalgebras $S\gamma$.
		\end{itemize}
		\item There is a pseudonatural transformation $\varepsilon: F^{S} U^{S} \rightarrow 1_{\mathcal{A}^{S}}$ whose 
		\begin{itemize}
			\item $1$-cell component on a pseudoalgebra $\left(X, \phi, u, m\right)$ is the pseudomorphism $\left(\phi, m\right): \left(SX, \mu_{X}, \rho_{X}, \alpha_{X}\right) \rightarrow \left(X, \phi, u, m\right)$,
			\item $2$-cell component on a pseudomorphism $\left(f, \overline{f}\right): \left(X, \phi, u, m\right) \rightarrow \left(Y, \psi, v, n\right)$ is given by $\overline{f}$.
		\end{itemize}
		\item There is a modification $m: \varepsilon_{F}.F\eta \cong 1_{\varepsilon}$ whose component on the pseudoalgebra $\left(X, \phi, u, m\right)$ is given by $u$.
		\item The data described thus far define a pseudoadjunction in $\mathbf{Gray}$.
		\item Given a $2$-category $\mathcal{B}$, the canonical $2$-functor $K_\mathcal{B}: \mathbf{Gray}\left(\mathcal{B}, \mathcal{A}^{S}\right) \rightarrow {\mathbf{Gray}\left(\mathcal{B}, \mathcal{A}\right)}^{\mathbf{Gray}\left(\mathcal{B}, S\right)}$ is an isomorphism of $2$-categories.
		\item The assignment of $\mathcal{B}$ to $K_\mathcal{B}$ defines a $\mathbf{Gray}$-natural transformation.
	\end{enumerate}
\end{proposition}

\begin{proof}
	See Theorem 6.2 of \cite{Coherent Approach to Pseudomonads}.
\end{proof}

\begin{notation}
	The $2$-category $\mathbf{FreePsAlg}\left(\mathcal{A}, S\right)$ will denote the full sub-$2$-category of $\mathcal{A}^S$ on those pseudoalgebras of the form $\left(SX, \mu_{X}, \rho_{X}, \alpha_{X}\right)$ for some $X \in \mathcal{A}$.
\end{notation}

\subsection{Kleisli bicategories}\label{section Kleisli Bicategory}

\noindent We review the construction of the Kleisli bicategory of a pseudomonad given in \cite{Cheng Hyland Power Pseudodistributive Laws}. Even when the pseudomonad being considered is in $\mathbf{Gray}$, and hence the base $\mathcal{A}$ and the endomorphism $S: \mathcal{A} \rightarrow \mathcal{A}$ are strict, this type of Kleisli construction only produces a bicategory. Hence $\mathbf{Gray}$-enriched notions are unable to express its universal property. We do not need the Kleisli bicategory for any of the results in this paper. Rather, once our results are built upon further in \cite{Miranda Tricategorical Limits and Colimits}, a deeper understanding of its universal property can be achieved. This is desirable, since as we mentioned in the introduction, Kleisli bicategories occur more naturally than the $\mathbf{Gray}$-enriched Kleisli objects that we will describe in Subsection \ref{Subsection the presentation of the Kleisli object of a pseudomonad}.
\\
\\
\noindent Since it has many parts, we summarise the content of Proposition \ref{Kleilsi bicategory}.
\begin{itemize}
	\item Part (1) recalls the bicategory structure,
	\item Parts (2) and (3) recall the left and right pseudoadjoints respectively,
	\item Parts (4) and (5) respectively describe the counit and left triangulator of the pseudoadjunction, which is collected in part (6),
	\item Part (7) records that the pseudoadjunction only splits the original pseudomonad up to an invertible icon,
	\item Part (8) recalls the universal property.
\end{itemize}   

\begin{proposition}\label{Kleilsi bicategory}
	Let $\left(\mathcal{A}, S\right)$ a pseudomonad in $\mathbf{Gray}$.
	
	\begin{enumerate}
		\item There is a bicategory $\mathbf{KL}\left(\mathcal{A}, S\right)$ with the following data. 
		
		\begin{itemize}
			\item Objects as in $\mathcal{A}$,
			\item Hom-categories given by $\mathcal{A}\left(X, SY\right)$,
			\item Units given by $\eta_{X}: X \rightarrow SX$,
			\item Composition functors given by
		\end{itemize}
		$$\begin{tikzcd}
			\mathcal{A}\left(Y{,} SZ\right) \times \mathcal{A}\left(X{,} SY\right) \arrow[r, "S \times 1"]
			& \mathcal{A}\left(SY{,} S^2Z\right) \times \mathcal{A}\left(X{,} SY\right) \arrow[r, "\circ"]
			& \mathcal{A}\left(X{,} S^{2}Z\right) \arrow[r, "\mathcal{A}\left(X{,}\mu_{Z}\right)"]
			& \mathcal{A}\left(X{,} SZ\right) 
		\end{tikzcd}$$
		\begin{itemize}
			\item Left and right unitors on $f: X \rightarrow SY$ given by the pastings in $\mathcal{A}$ depicted below. 
			
			$$
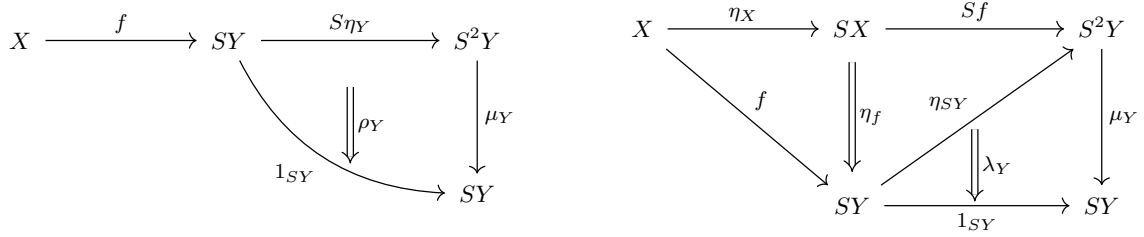
\begin{tikzcd}[font=\fontsize{9}{6}]
				X \arrow[rr, "f"]
				&& SY \arrow[rr, "S\eta_{Y}"] \arrow[rrdd,bend right =30, "1_{SY}"']
				&{}\arrow[dd,Rightarrow, shorten = 10, "\rho_{Y}"]& S^{2}Y \arrow[dd, "\mu_{Y}"]
				\\
				\\
				&&&
				{}
				&SY&{}
			\end{tikzcd}
			\begin{tikzcd}[font=\fontsize{9}{6}]
				X \arrow[rrdd, "f"]
				\arrow[rr, "\eta_{X}"]
				&& SX
				\arrow[dd, Rightarrow, shorten = 5, "\eta_{f}"]
				\arrow[rr, "Sf"] 
				&& S^{2}Y \arrow[dd, "\mu_{Y}"]
				\\
				&&&{}\arrow[d, Rightarrow, "\lambda_{Y}"]
				\\
				&&
				SY
				\arrow[rr, "1_{SY}"']
				\arrow[rruu, "\eta_{SY}"]
				&{}&SY
			\end{tikzcd}$$
			\item Associator given by the pasting in $\mathcal{A}$ depicted below. 
			
			$$
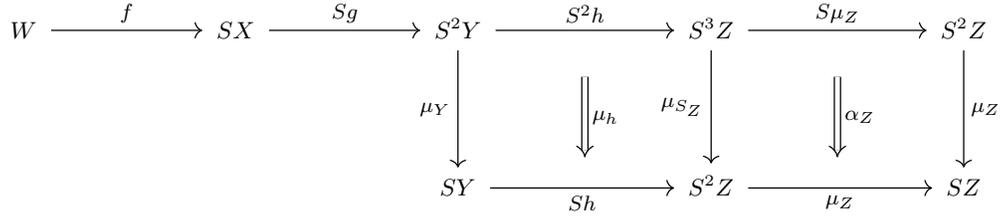
\begin{tikzcd}[font=\fontsize{9}{6}]
				W \arrow[rr, "f"]
				&& SX \arrow[rr, "Sg"]
				&& {S^2}Y \arrow[rr,"{S^2}h"]
				\arrow[dd, "{\mu}_Y"']
				&{}\arrow[dd,Rightarrow, shorten = 10, "\mu_{h}"]
				&{S^3}Z \arrow[rr, "S {\mu}_Z"]
				\arrow[dd, "{\mu}_{{S}_{Z}}"']
				&{}\arrow[dd, Rightarrow, shorten = 10, "\alpha_{Z}"]& {S^2}Z \arrow[dd, "{\mu}_{Z}"]
				\\
				\\
				&&&& SY \arrow[rr, "Sh"']
				&{}& {S^2}Z \arrow[rr, "{\mu}_Z"']
				&{}& SZ
			\end{tikzcd}$$
		\end{itemize}
		\item There is a normal pseudofunctor $L: \mathcal{A} \rightsquigarrow \mathbf{KL}\left(\mathcal{A}, S\right)$ with the following description.
		\begin{itemize}
			\item It is given by the identity on objects.
			\item It is given by composition with $\eta_{Y}$ between hom-categories:
			
			$$\begin{tikzcd}
				\mathcal{A}\left(X{,}Y\right) \arrow[rr, "\mathcal{A}\left(X{,}\eta_{Y}\right)"] && \mathcal{A}\left(X{,}SY\right)
			\end{tikzcd}$$
			\item The compositor on \begin{tikzcd}
				X \arrow[r, "f"] & Y \arrow[r, "g"] &Z
			\end{tikzcd} is given by the $2$-cell in $\mathbf{KL}\left(\mathcal{A}, S\right)$ determined by the pasting in $\mathcal{A}$ depicted below.
			
			$$
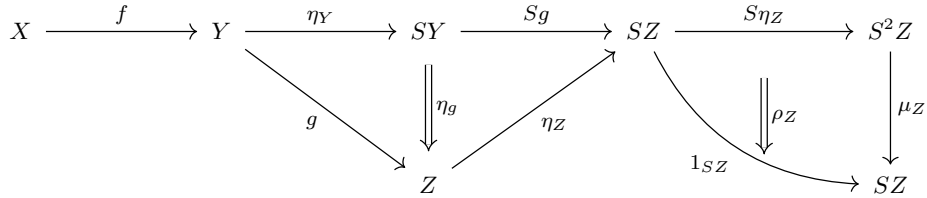
\begin{tikzcd}[font=\fontsize{9}{6}]
				X \arrow[rr, "f"] 
				&& Y\arrow[rrdd, "g"'] 
				\arrow[rr, "\eta_{Y}"]
				&& SY \arrow[rr, "Sg"]
				\arrow[dd, Rightarrow, shorten = 5, "\eta_{g}"]
				&& SZ \arrow[rrdd, bend right = 30, "1_{SZ}"']
				\arrow[rr, "S\eta_{Z}"]
				&{}\arrow[dd, Rightarrow, shorten = 10, "\rho_{Z}"]
				& S^{2}Z\arrow[dd, "\mu_{Z}"]
				\\
				\\
				&&&&
				Z\arrow[rruu, "\eta_{Z}"']
				&&&{}& SZ
			\end{tikzcd}$$
		\end{itemize}
		\item There is a pseudofunctor $R: \mathbf{KL}\left(\mathcal{A}, S\right) \rightsquigarrow \mathcal{A}$ with the following description.
		
		\begin{itemize}
			\item It is given on objects by sending $X$ to $SX$.
			\item It is given between hom-categories as depicted below.
			
			$$\begin{tikzcd}
				\mathcal{A}\left(X{,}SY\right) \arrow[rr, "S"] && \mathcal{A}\left(SX{,}S^2Y\right) \arrow[rr, "\mathcal{A}\left(X{,}\mu_{Y}\right)"]
				&& \mathcal{A}\left(SX{,}SY\right)
			\end{tikzcd}$$
			\item The component of its unitor on $X$ is given by $\rho_{X}$
			\item The component of its compositor on $\left(f: X \rightarrow SY, g: Y \rightarrow SZ\right)$ is given by the pasting depicted below.
			$$\begin{tikzcd}[font=\fontsize{9}{6}]
				SX \arrow[rr, "Sf"] 
				&& S^2Y
				\arrow[rrdd, "S^2g"']
				\arrow[rr, "\mu_{Y}"] 
				&& SY
				\arrow[dd, Rightarrow, shorten = 8, "\mu_{g}"]
				\arrow[rr, "Sg"]
				&&S^2Z
				\arrow[rr, "\mu_{Z}"]
				\arrow[dd, Rightarrow, shorten = 8, "\alpha_{Z}"]
				&&SZ
				\\
				\\
				&&&&S^3Z
				\arrow[rruu, "\mu_{SZ}"]
				\arrow[rr, "S\mu_{Z}"']
				&{}
				&	S^2Z
				\arrow[rruu, "\mu_{Z}"']
			\end{tikzcd}$$		
			
		\end{itemize}
		\item There is a pseudonatural transformation $\varepsilon: LR \Rightarrow 1_{\mathbf{KL}\left(\mathcal{A}, S\right)}$ admitting the following description.
		\begin{itemize}
			\item Its $1$-cell component on $X$ given by $1_{SX}: SX \rightarrow SX$, seen as a morphism in $\mathbf{KL}\left(\mathcal{A}, S\right)$,
			\item Its $2$-cell component on $f: X \rightarrow SY$ given by the $2$-cell in $\mathbf{KL}\left(\mathcal{A}, S\right)$ determined by the left whiskering of $\lambda_{Y}$ with $Rf$ in $\mathcal{A}$.
		\end{itemize} 
		\item There is an invertible modification $\left(R\varepsilon\right).\left(\eta R\right) \Rrightarrow 1_{R}$, whose component on $X$ is given by the $2$-cell $\lambda_{X}: \mu_{X}.\eta_{SX} \Rightarrow 1_{SX}$.
		\item The data described in (1)-(5) define a pseudoadjunction in $\mathbf{Bicat}$.
		\item The pseudomonad induced on $\mathcal{A}$ is isomorphic to $\left(\mathcal{A}, S\right)$ via an invertible icon $\phi: S \rightarrow RL$ whose $2$-cell component on $f: X \rightarrow Y$ is given by the left whiskering of $\rho_{Y}$ with $f$.
		\item For any bicategory $\mathcal{B}$, precomposition with $L$ induces a biequivalence $\mathbf{Bicat}\left(\mathbf{KL}\left(\mathcal{A}, S\right), \mathcal{B}\right) \rightsquigarrow {\mathbf{Bicat}\left(\mathcal{A}, \mathcal{B}\right)}^{\mathbf{Bicat}\left(S, \mathcal{B}\right)}$, where the target is the bicategory of pseudoalgebras, pseudomorphisms and $2$-cells for the pseudomonad determined by restriction along $S$.
	\end{enumerate} 
\end{proposition}

\begin{proof}
	The proof is via a direct calculation. See \cite{Cheng Hyland Power Pseudodistributive Laws} Definition 4.1 for the definition of the Kleisli bicategory, Definition 4.2 for a description of the bicategory of pseudoalgebras ${\mathbf{Bicat}\left(\mathcal{A}, \mathcal{B}\right)}^{\mathbf{Bicat}\left(S, \mathcal{B}\right)}$ via cocones, Theorem 4.3 for the universal property in part (8), and Corollary 4.4 for the pseudoadjunction in part (6).
\end{proof}

\noindent In the forthcoming \cite{Miranda Tricategorical Limits and Colimits} we will show that the biequivalences of Proposition \ref{Kleilsi bicategory} part (8) are trinatural in $\mathcal{B}$, and that Kleisli bicategories have a universal property of a tricategorical colimit. This will answer the question posed in Remark 4.7 of \cite{Formal Theory of Pseudomonads}.

\begin{remark}\label{combinatorics of Kleisli arrows and composition}
	We highlight what we consider to be the most combinatorially interesting aspect of the proof of Proposition \ref{Kleilsi bicategory}, namely the proof of the pentagon axiom for the bicategory structure on $\mathbf{KL}\left(\mathcal{A}, S\right)$. In the diagram below, we have omitted notation for the $2$-cells which fill each square. These are all components of $\mu$ or $\alpha$. The boundary corresponds to the pentagon of associators. The two cubes are the two steps of the proof, with the right cube commuting by the associativity axiom for $\left(\mathcal{A}, S\right)$ on $Z$, and the left cube commuting by the modification condition for $\alpha$ on the morphism $z: Y \rightarrow SZ$.
	
	$$\begin{tikzcd}[column sep=12,row sep=12, font=\fontsize{9}{6}]
		V \arrow[rrrr, "w"]
		&&&& SW \arrow[rrrr, "Sx"]
		&&&& {S^2}X \arrow[rrrr, "{S^2}y"] \arrow[dddd, "{\mu}_X"']
		&&&& {S^3}Y
		\arrow[rrrr, "{S^3}z"]
		\arrow[rddd, "S{{\mu}_Y}"]
		\arrow[dddd, "{\mu}_{SY}"']
		&&&& {S^4}Z 
		\arrow[rrrr, "{S^2}{{\mu}_Z}"]
		\arrow[rddd,"S{{\mu}_{SZ}}"]
		\arrow[dddd, "{\mu}_{{S^2Z}}"', near start]
		&&&& {S^3}Z
		\arrow[rddd, "S{{\mu}_Z}"]
		\arrow[dddd, "{\mu}_{SZ}"', near start]
		\\
		\\
		\\
		&&&&&&&&&&&&& {S^2}Y \arrow[rrrr, "{S^2}z", near start] 
		\arrow[dddd, "{\mu}_Y", near end]
		&&&& {S^3}Z \arrow[rrrr, "S{{\mu}_Z}", near start] 
		\arrow[dddd, "{\mu}_{SZ}", near end]
		&&&& {S^2}Z \arrow[dddd, "{\mu}_Z", near end]
		\\
		&&&&&&&& SX \arrow[rrrr, "Sy"']
		&&&& {S^2}Y \arrow[rrrr, "{S^2}z"', near end] \arrow[rddd, "{\mu}_Y"']
		&&&& {S^3}Z \arrow[rrrr, "{S}{{\mu}_Z}"', near end] \arrow[rddd, "{\mu}_{SZ}"']
		&&&& {S^2}Z \arrow[rddd, "{\mu}_Z"']
		\\
		\\
		\\
		&&&&&&&&&&&&& SY \arrow[rrrr, "Sz"']
		&&&& {S^2}Z \arrow[rrrr, "{\mu}_z"']
		&&&& SZ
	\end{tikzcd}$$
\end{remark}

\section{Weights and presentations}\label{Section weights and presentations}

\subsection{Kleisli categories for monads presented via formal algebra structures}\label{Subsection Kleisli via formal algebra structures}

\noindent Let $\mathbf{n}$ denote the ordered set $\{1 \leq ...\leq n\}$, and let $\mathbf{0}$ denote the empty set. Let $\Delta_{+}$ denote the category of finite ordinals and monotone maps. Ordinal sum $\left(\mathbf{m}, \mathbf{n}\right)\mapsto \mathbf{m}\oplus \mathbf{n}$ endows this category with a monoidal product, and the canonical maps \begin{tikzcd}
	\mathbf{0} \arrow[r, "!"] & \mathbf{1} &\mathbf{2}\arrow[l, "!"']
\end{tikzcd} endow the terminal object with a monoid structure. Moreover, $\Delta_{+}$ is the free strict monoidal category containing a monoid object. Let $\mathbf{Mnd}$ denote the suspension of $\Delta_{+}$, as a $2$-category with one object. A \emph{monad} in a $2$-category $\mathcal{K}$ is a $2$-functor $\mathbb{S}: \mathbf{Mnd} \rightarrow \mathcal{K}$. We say this monad is on the object $\mathbb{S}\left(\star\right)$, where $\star$ denotes the unique object of $\mathbf{Mnd}$.

\begin{example}\label{Example weight for ordinary Kleisli}
	Let $\Delta_{\bot}$ denote the sub-category of $\Delta_{+}$ consisting of those ordinals which have a bottom element, and of those monotone maps which preserve the bottom element. There is a $2$-functor $\mathbf{M}: \mathbf{Mnd} \rightarrow \mathbf{Cat}$ corresponding to the monad $ \mathbf{1}\oplus -$ on $\Delta_{\bot}$ given by freely adjoining a bottom element.
\end{example}

\noindent Let $\left(A, S, \eta, \mu\right)$ be a monad in $\mathbf{Cat}$ and let $\mathbb{S}: \mathbf{Mnd} \rightarrow \mathbf{Cat}$ denote the corresponding $2$-functor. The Kleisli category of a monad $\left(A, S, \eta, \mu\right)$ is the colimit of $\mathbb{S}$ weighted by the $2$-functor $\mathbb{M}: \mathbf{Mnd}^\text{op} \rightarrow \mathbf{Cat}$ described in Example \ref{Example weight for ordinary Kleisli} \cite{Formal Theory of Monads}. We take this as the definition of the Kleisli category $A_S$. Unwinding this definition yields the following coequaliser in $\mathbf{Cat}$.

$$\begin{tikzcd}
	{\Delta}_{\bot} \times \Delta_{+} \times A \arrow[rr, shift right = 2, "1_{\Delta_{\bot}}\times \mathbb{S}"']
	\arrow[rr, shift left = 2, "\mathbb{M}\times 1_{A}"] && \Delta_{\bot} \times A \arrow[rr, "\mathbb{F}"] && A_{S} 
\end{tikzcd}$$

\noindent However, by the coherence theorem for algebras over monoids, the remaining data of a functor $\mathbb{F}': \Delta_{\bot}\times A \rightarrow B$ satisfying $\mathbb{F}'.\left(\mathbb{M}\times 1_{A}\right) = \mathbb{F}'.\left(1_{\Delta_\bot}\times \mathbb{S}\right)$ is determined by the following information.

\begin{itemize}
	\item A functor $F': A \rightarrow B$,
	\item A natural transformation $\varepsilon': F'S \rightarrow F'$,
\end{itemize}

\noindent Such that $\left(F', \varepsilon'\right)$ is an algebra for the monad on the functor category $[A, A_{S}]$ given by restriction along $\left(S, \eta, \mu\right)$. Being such an algebra is to say that the following equations of natural transformations hold.

$$\begin{tikzcd}[font=\fontsize{9}{6}]
	F'S^{2} \arrow[rr, "F'\mu"]\arrow[dd, "\varepsilon' S"'] && F'S\arrow[dd, "\varepsilon'"]
	\\
	\\
	F'S \arrow[rr, "\varepsilon'"'] && F'&{}
\end{tikzcd}\begin{tikzcd}[font=\fontsize{9}{6}]
F' \arrow[rrdd, bend right = 30, "1"']
\arrow[rr, "F'.\eta"] && F'S\arrow[dd, "\varepsilon'"]
\\
\\
&& F'
\end{tikzcd}$$

\noindent A presentation for the Kleisli category can be deduced from this universal property. The objects in $A_S$ are the same as in $A$, and for every morphism $f: X \rightarrow Y$ in $A$, there is a generating morphism with the same domain and codomain in $A_S$. Moreover, for every $X \in A$, there is also a generator of the form $\varepsilon_{X}: SX \rightarrow X$. These data are subject to relations corresponding to relations between morphisms in $A$, as well as three new relations. These relations pertain to naturality of $\varepsilon$ in each generator of the form $f \in A\left(X, Y\right)$ as depicted below, and the two algebra axioms listed above.

$$\begin{tikzcd}[font=\fontsize{9}{6}]
	SX \arrow[rr, "Sf"]\arrow[dd, "\varepsilon_{X}"'] && SY\arrow[dd, "\varepsilon_{Y}"]
	\\
	\\
	X \arrow[rr, "f"'] && Y
\end{tikzcd}$$

\noindent The category defined via this presentation is indeed isomorphic to the more familiar version of the Kleisli category, consisting of Kleisli arrows. Letting $\mathbf{KL}\left(A, S\right)$ denote the Kleisli category whose morphisms are of the form $f: X \rightarrow SY$ with units $\eta_{X}$ and composition $g\circ_\mathbb{S}f : = \mu_{Z}.Sg.f$, the isomorphism of categories $I: A_{S} \rightarrow \mathbf{KL}\left(A, S\right)$ is given by the identity on objects, $\left(f: X \rightarrow Y\right) \mapsto \eta_{Y}.f$ for $f$ a generator from $A$, and $\left(\varepsilon_{X}: SX \rightarrow X\right) \mapsto 1_{SX}$.

\subsection{Kleisli objects for pseudomonads as weighted colimits}\label{Section definition and presentation of Gray enriched Kleisli objects for pseudomonads}
\noindent The following definitions are discussed and justified in \cite{Coherent Approach to Pseudomonads}. Here $\otimes$ denotes the Gray tensor product.

\begin{definition}\label{free living pseudomonad and pseudpadjunction}
	This definition recalls the free $\mathbf{Gray}$-monoid containing a pseudomonoid $\overline{\Delta}_{+}$, with the free $\mathbf{Gray}$-category containing a pseudomonad being its suspension.
	\begin{itemize}
		\item Let $\bar{\Delta}_{+}$ be the $2$-category in which \begin{itemize}
			\item The objects are the natural numbers,
			\item The morphisms are freely generated by the generating coface and codegeneracy simplicial maps $d_{0}, ... , d_{n}: [n] \rightarrow [n+1]$ and $s_{0}, ...., s_{n}: [n+1] \rightarrow [n]$,
			\item A $2$-cell $\phi \Rightarrow \psi$ exists whenever the paths $\phi$ and $\psi$ compose to give equal morphisms in the usual category of simplices $\Delta_{+}$. Note that these will necessarily be invertible.
		\end{itemize} 
		\item Let $\oplus: \bar{\Delta}_{+} \otimes \bar{\Delta}_{+} \rightarrow \bar{\Delta}_{+}$ denote the ordinal sum described in section 2 of \cite{Coherent Approach to Pseudomonads}.
		\item Let $\mathbf{Psmnd}$ be the suspension of the $\mathbf{Gray}$-monoid $\bar{\Delta}_{+}$.
		\item Let $\mathbf{Psadj}$ be the $\mathbf{Gray}$-category similarly constructed from the free-living adjunction $\mathbf{Adj}$ of \cite{Free Adjunction}, via a factorisation as described in section 5.1 of \cite{Coherent Approach to Pseudomonads} and recalled below. In this factorisation
		\begin{itemize}
			\item $\mathbb{ADJ}$ is a $2$-computad \cite{Street Limits Indexed By Category Values 2-functors} with two objects $X, Y$, two edges $f: X \rightarrow Y$ and $u: Y \rightarrow X$, and two $2$-cells $\eta: 1_{X} \Rightarrow uf$ and $\varepsilon: fu \Rightarrow 1_{Y}$. 
			\item $\mathcal{F}_2\left(-\right)$ denotes the free $2$-category of a $2$-computad and $D$ takes the discrete $\mathbf{Gray}$-category of a sesquicategory.
			\item $B$ is bijective on $k$-cells for $0 \leq k \leq 2$, while $F$ is fully faithful on $3$-cells.
		\end{itemize}
		
		$$\begin{tikzcd}
			\mathcal{F}_2\left(\mathbb{ADJ}\right)
			\arrow[rr, "B"]
			&& \mathbf{Psadj}
			\arrow[rr, "F"] 
			&& D\left(\mathbf{Adj}\right)
		\end{tikzcd}$$
		
		\item Let $J: \mathbf{Psmnd} \rightarrow \mathbf{Psadj}$ be the full-sub-$\mathbf{Gray}$-category inclusion on the object $\mathbf{m}$ on which the pseudomonad lives and let $\mathbf{c}$ denote the other object.
		\item A \emph{pseudomonad} in a $\mathbf{Gray}$-category $\mathfrak{K}$ will be a $\mathbf{Gray}$-functor $\mathbb{S}: \mathbf{Psmnd} \rightarrow \mathfrak{K}$.
	\end{itemize}
	
\end{definition}

\begin{remark}\label{remark sufficient data for a pseudomonad}
	By Theorem 3.5 of \cite{Coherent Approach to Pseudomonads}, if $\mathfrak{K}$ is a $\mathbf{Gray}$-category then to specify a pseudomonad $\mathbb{S}: \mathbf{Psmnd} \rightarrow \mathfrak{K}$ it suffices to do the following
	
	\begin{enumerate}
		\item Specify the images of
		
		\begin{itemize}
			\item The unique object $\mathbf{m}$ of $\mathbf{Psmnd}$. This will be referred to as the \emph{base} of the pseudomonad $\mathbb{S}$.
			\item The object $\mathbf{1} \in \bar{\Delta}_{+}$, viewed as an arrow in $\mathbf{Psmnd}$. This will be referred to as the \emph{underlying endomorphism} of the pseudomonad $\mathbb{S}$.
			\item The unique arrows $d_{0}: \mathbf{0} \rightarrow \mathbf{1}$ and $s_{0}: \mathbf{2} \leftarrow \mathbf{1}:$ in $\bar{\Delta}_{+}$, viewed as $2$-cells in $\mathbf{Psmnd}$. Then $\mathbb{S}d_{0}$ and $\mathbb{S}s_{0}$ will be referred to as the \emph{unit} and \emph{multiplication} of $\mathbb{S}$ respectively.
			\item The following invertible $2$-cells in $\bar{\Delta}_{+}$, each viewed as $3$-cells in $\mathbf{Psmnd}$ \begin{itemize}
				\item $1_\mathbf{1} \Rightarrow s_{0} . \left(d_{0} \oplus \mathbf{1}\right)$,
				\item $ s_{0}.\left(\mathbf{1}\oplus d_{0}\right) \Rightarrow 1_\mathbf{1}$, and
				\item $ s_{0}.\left(\mathbf{1}\oplus s_{0}\right) \Rightarrow s_{0}.\left(s_{0} \oplus \mathbf{1}\right)$.
			\end{itemize}
			The images of this data under $\mathbb{S}$ will be respectively referred to as the \emph{left unitor}, \emph{right unitor}, and \emph{associator} of $\mathbb{S}$. 
		\end{itemize}
		\item Check that coherences 1 and 2 as in Definition 3.1 of \cite{Marmolejo Pseudodistributive 1}, hold for these data. These coherence conditions are recalled below, using the more familiar $2$-categorical notation introduced in Notation \ref{pseudomonad notation}.
	\end{enumerate}
	Indeed, what we have described in this remark is usually what is understood by a pseudomonad and it is the content of Theorem 3.5 of \cite{Coherent Approach to Pseudomonads} that this data uniquely determines the rest of the $\mathbf{Gray}$-functor $\mathbb{S}:\mathbf{Psmnd} \rightarrow \mathfrak{K}$. We choose to take the $\mathbf{Gray}$-functor formulation as primitive since we will be considering its weighted limits and colimits.

	$$\begin{tikzcd}[font=\fontsize{9}{6}]
		S^{4} \arrow[dd, "\mu S^{2}"']
		\arrow[rr, "S^{2}\mu"]\arrow[rd, "S\mu S"description]
		&{}\arrow[d, Leftarrow, shorten = 5, "S\alpha"]& S^{3}\arrow[rd, "S\mu"]
		\\
		&S^{3}\arrow[rr, "S\mu"]\arrow[dd, "\mu S"]\arrow[d, Leftarrow, shorten = 5, shift right = 12, "\alpha S"]
		&{}\arrow[dd, Leftarrow, shorten = 15, "\alpha"]
		&S^2\arrow[dd, "\mu"]
		\\
		S^{3}\arrow[rd, "\mu S"']&{}&&&=
		\\
		&S^2 \arrow[rr, "\mu"']
		&{}& S
	\end{tikzcd}\begin{tikzcd}[font=\fontsize{9}{6}]
		S^{4} \arrow[dd, "\mu S^{2}"']
		\arrow[rr, "S^{2}\mu"]
		&{}\arrow[dd, Leftarrow, shorten = 20, Leftarrow, "\mu_{\mu}"]
		& S^{3}\arrow[rd, "S\mu"]\arrow[dd, "\mu S"description]
		\\
		&&&S^2\arrow[dd, "\mu"]
		\arrow[d, Leftarrow, shorten = 5, shift right = 12, "\alpha"]
		\\
		S^{3}\arrow[rd, "\mu S"']\arrow[rr, "S\mu"]
		&{}\arrow[d, Leftarrow,"\alpha"]
		&S^2\arrow[rd, "\mu"]
		&{}
		\\
		&S^2 \arrow[rr, "\mu"']
		&{}& S
	\end{tikzcd}$$
	
	$$\begin{tikzcd}[font=\fontsize{9}{6}]
		&&{}\arrow[dd, Rightarrow, shorten = 10, "\rho S"]
		&&S^{2}\arrow[rrdd, "\mu"]\arrow[dddd, Rightarrow, shorten = 15, "\alpha"]
		\\
		\\
		S^{2} \arrow[rrrruu, bend left = 30, "1_{S^2}"]
		\arrow[rrrrdd, bend right = 30, "1_{S^{2}}"']
		\arrow[rr, "S\eta S"]
		&& S^{3} \arrow[rruu, "\mu S"]\arrow[rrdd, "S\mu"]\arrow[dd, Rightarrow, shorten = 10, "S\lambda"]
		&&&&S&=&S^{2}\arrow[rr, bend left = 20, "\mu"name=A]\arrow[rr, bend right = 20, "\mu"'name=B] && S
		\\
		\\
		&&{}&&S^{2}\arrow[rruu, "\mu"]
		\arrow[from=A, to=B, Rightarrow, shorten = 5, "1_{\mu}"]
	\end{tikzcd}$$
\end{remark}

\begin{definition}\label{Gray Kleisli def as a left kan extension}
	Let $\mathfrak{K}$ be a $\mathbf{Gray}$-category and let $\mathbb{S}: \mathbf{Psmnd} \rightarrow \mathfrak{K}$ be a pseudomonad. 
	
	\begin{itemize}
		\item The \emph{$\mathbf{Gray}$-enriched Eilenberg-Moore pseudoadjunction} associated to $\mathbb{S}$ will be, if it exists, the right Kan extension of $\mathbb{S}$ along $J: \mathbf{Psmnd} \rightarrow \mathbf{Psadj}$. The image under $\mathbf{Ran}_{J}\mathbb{S}$ of $\mathbf{c}$ will be called the \emph{$\mathbf{Gray}$-categorical Eilenberg-Moore object of $\mathbb{S}$} and will be denoted as $A^S$ where $A:={\mathbb{S}\left(\mathbf{m}\right)}$ and $S:={\mathbb{S}\mathbf{1}}$.
		\item The \emph{$\mathbf{Gray}$-enriched Kleisli pseudoadjunction} associated to $\mathbb{S}: \mathbf{Psmnd} \rightarrow \mathfrak{K}$ will be, if it exists, the left Kan extension of $\mathbb{S}$ along $J: \mathbf{Psmnd} \rightarrow \mathbf{Psadj}$. The image under $\mathbf{Lan}_{J}\mathbb{S}$ of $\mathbf{c}$ will be called the \emph{$\mathbf{Gray}$-categorical Kleisli object of $\mathbb{S}$} and will be denoted ${\mathbb{S}\left(\mathbf{m}\right)}_{\mathbb{S}\mathbf{1}}$.
	\end{itemize}
\end{definition}

\begin{notation}\label{pseudomonad notation}
	When the rest of the data of a pseudomonad $\mathbb{S}$ is clear from context, $\mathbb{S}$ will be denoted by $\left(\mathbb{S}\left(\mathbf{m}\right), \mathbb{S}\mathbf{1}\right)$. In most of what follows, we will be considering a pseudomonad in $\mathbf{Gray}$ for which we fix the following notation.
	
	\begin{itemize}
		\item  The base of our pseudomonad will be a $2$-category $\mathcal{A}$, 
		\item The underlying endomorphism will be a $2$-functor $S: \mathcal{A} \rightarrow \mathcal{A}$, 
		\item The unit and multiplication will be pseudonatural transformations $\eta: 1_\mathcal{A} \Rightarrow S$ and $\mu: S^2 \Rightarrow S$ respectively,
		\item The left and right unitors will be invertible modifications $\lambda: 1_{S} \Rrightarrow \mu. \eta_{S}$ and $\rho: \mu.S\eta \Rrightarrow 1_{S}$ respectively,
		\item The associator will be an invertible modification $\alpha: \mu.S\mu \Rrightarrow \mu.\mu_{S}$.
	\end{itemize}  
\end{notation}

\begin{remark}\label{Gray Kleilsi's universal cocone as a weighted colimit}
	Given a pseudomonad $\left(\mathcal{A}, S\right): \mathbf{Psmnd} \rightarrow \mathbf{Gray}$, we can re-express the universal properties of $\mathcal{A}^S$ and $\mathcal{A}_S$ in terms of weighted limits and colimits, and describe their universal cones and cocones.
	
	\begin{itemize}
		\item The $\mathbf{Gray}$-categorical Eilenberg Moore object $\mathcal{A}^S$ can equivalently be described as the $\mathbf{Gray}$-categorical limit of $\left(\mathcal{A}, S\right)$ weighted by the $\mathbf{Gray}$-functor displayed below.
		
		$$\begin{tikzcd}
			\mathbf{Psmnd} \arrow[rr, "J"] && \mathbf{Psadj} \arrow[rr, "\mathbf{Psadj}\left(\mathbf{c}{,} -\right)"] &&\mathbf{Gray}
		\end{tikzcd}$$
		\item The $\mathbf{Gray}$-categorical Kleisli object $\mathcal{A}_S$ can equivalently be described as the $\mathbf{Gray}$-categorical colimit of $\left(\mathcal{A}, S\right)$ weighted by the $\mathbf{Gray}$-functor displayed below.
		$$\begin{tikzcd}
			\mathbf{Psmnd}^\text{op} \arrow[rr, "J^\text{op}"] && \mathbf{Psadj}^\text{op} \arrow[rr, "\mathbf{Psadj}\left(-{,} \mathbf{c}\right)"] &&\mathbf{Gray}
		\end{tikzcd}$$ 
	\end{itemize} 
	As per Proposition \ref{EM universal property}, $\mathcal{A}^S$ is precisely the $2$-category of pseudoalgebras, pseudomorphisms and pseudoalgebra $2$-cells described in Section \ref{Subsection free pseudoalgebras}. Now for $\mathcal{A}_S$ we have for every $\mathcal{B} \in \mathbf{Gray}$ the universal property exhibiting isomorphism of $2$-categories depicted below, which varies $\mathbf{Gray}$-naturally in $\mathcal{B}$.
	\\
	$$ \Phi: \mathbf{Gray}\left(\mathcal{A}_S {,} \mathcal{B}\right) \rightarrow {[}\mathbf{Psmnd}{,}\mathbf{Gray}{]}\big(\mathbf{Psadj}\left(J-{,}\mathbf{c}\right){,}\mathbf{Gray}\left(\left(\mathcal{A}{,}S\right)-{,}\mathcal{B}\right)\big)$$

\end{remark}

\noindent As observed in Remark 4.7 of \cite{Formal Theory of Pseudomonads}, if one attempts to take $\mathcal{A}_{S}$ to be free pseudoalgebras then the appropriate universal property does not hold. The analogous $2$-functor is only a weak biequivalence in the sense of Definition \ref{definition weak biequivalence}, to follow.

\begin{definition}\label{definition weak biequivalence}
	Call a $2$-functor $F: \mathcal{A} \rightarrow \mathcal{B}$ a \emph{weak biequivalence} if it is biessentially surjective on objects, and for each $X, Y \in \mathcal{A}$ the functor $F_{X, Y}: \mathcal{A}\left(X, Y\right) \rightarrow \mathcal{B}\left(FX, FY\right)$ is an equivalence of categories.
\end{definition}

\subsection{The presentation of the Kleisli object of a pseudomonad in $\mathbf{Gray}$}\label{Subsection the presentation of the Kleisli object of a pseudomonad}

	\noindent By the coherence theorem for pseudoalgebras, pseudomorphisms, and pseudoalgebra $2$-cells (Proposition 4.2 of \cite{Coherent Approach to Pseudomonads}) the codomain of $\Phi$ can be identified with $ {\mathbf{Gray}\left(\mathcal{A}{,}\mathcal{B}\right)}^{\mathbf{Gray}\left(S{,} \mathcal{B}\right)}$. Taking $\mathcal{B}= \mathcal{A}_S$ we can describe the universal cocone determining the $\mathbf{Gray}$-enriched Kleisli object of $\left(\mathcal{A}, S\right)$. This cocone is the image under $\Phi$ of the identity $2$-functor on $\mathcal{A}_S$. Hence this universal cocone consists of the $2$-category $\mathcal{A}_S$ and a pseudoalgebra for the pseudomonad $\left(\mathbf{Gray}\left(\mathcal{A}, \mathcal{A}_S\right), \mathbf{Gray}\left(S, \mathcal{A}_S\right)\right)$. Note that by contravariance of $\mathbf{Gray}\left(-, \mathcal{A}_{S}\right)$, the roles of $\lambda$ and $\rho$ are swapped.
	\\
	\\
	\noindent The pseudoalgebra corresponding to the universal cocone consists of a $2$-functor $F: \mathcal{A} \rightarrow \mathcal{A}_S$, a pseudonatural transformation $\varepsilon: FS \Rightarrow F$ and two invertible modifications $u$ and $m$ as depicted below.
	
	$$
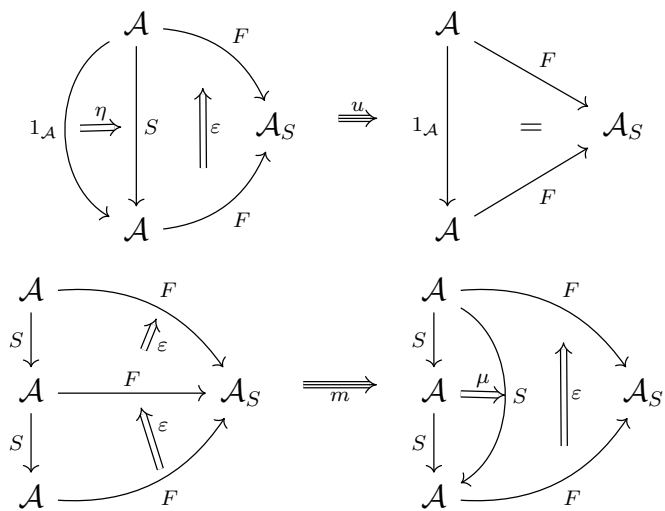
\begin{tikzcd}[column sep= 15]
		\mathcal{A} \arrow[dd,"1_\mathcal{A}"'{name=A},bend right =60] \arrow[dd,"S"{name=B}] \arrow[rrd,"F"{name=C},bend left] 
		\\
		&& \mathcal{A}_{S}
		\\
		\mathcal{A} \arrow[rru,"F"'{name=D},bend right]
		\arrow[Rightarrow,from=A,to=B,"\eta",shorten=2mm]
		\arrow[Rightarrow,from=D,to=C,"\varepsilon"',shorten=5mm,shift left=5]
	\end{tikzcd} \begin{tikzcd}[column sep = 15]
		{} \arrow[r,Rightarrow,"u"] \arrow[r, no head]
		& {}
	\end{tikzcd} \begin{tikzcd}[column sep = 12]
		\mathcal{A} \arrow[rrd,"F"] \arrow[dd,"1_\mathcal{A}"']
		\\
		& {=}
		& \mathcal{A}_{S}
		\\
		\mathcal{A} \arrow[rru,"F"']
	\end{tikzcd} $$
	
	$$\begin{tikzcd}
		\mathcal{A} \arrow[rrd,"F"{name=A},bend left] \arrow[d,"S"']
		\\
		\mathcal{A} \arrow[rr,"F"{name=B}] \arrow[d,"S"']
		&& \mathcal{A}_{S}
		\\
		\mathcal{A} \arrow[rru,"F"'{name=C},bend right]
		\arrow[Rightarrow,from=B,to=A,"\varepsilon"',shorten=2mm]
		\arrow[Rightarrow,from=C,to=B,"\varepsilon"',shorten=2mm]
	\end{tikzcd} \begin{tikzcd}
		{} \arrow[r,Rightarrow,"m"'] \arrow[r, no head]
		& {}
	\end{tikzcd} \begin{tikzcd}
		\mathcal{A} \arrow[d,"S"'{name=A}] \arrow[dd,bend left=60,"S"{name=B}] \arrow[rrd,"F"{name=C},bend left]
		\\
		\mathcal{A} \arrow[to=B,Rightarrow,"\mu"] \arrow[d,"S"']
		&& \mathcal{A}_{S}
		\\
		\mathcal{A} \arrow[rru,"F"'{name=D},bend right]
		\arrow[Rightarrow,from=D,to=C,"\varepsilon"',shift left=1,shorten=5mm]
	\end{tikzcd}$$
	
	\noindent The pseudoalgebra axioms then become the following equalities of $2$-cells in $\mathbf{Gray}\left(\mathcal{A}, \mathcal{A}_{S}\right)$.
	\\
	\\
	(Right unit coherence)
	\\
	\\
	\begin{tikzcd}[font=\fontsize{9}{6}]
		&&{}\arrow[d, Rightarrow, shorten = 5, "u_{S}"]
		&& FS \arrow[rrd,"\varepsilon"] \arrow[dd,Rightarrow,shorten=6mm,"m"]
		\\
		FS \arrow[rr,"F{\eta}_S"']
		\arrow[rrrru, bend left, "1_{FS}"]
		\arrow[rrrrd, bend right, "1_{FS}"']
		&& FS^2
		\arrow[d, shorten = 5, Rightarrow, "F\lambda"]
		\arrow[rru,"\varepsilon_{S}"] \arrow[rrd,"F_\mu"']
		&&&& F &&&=&1_\varepsilon
		\\
		&&{}
		&& FS \arrow[rru,"\varepsilon"']
	\end{tikzcd}
	\\
	\\
	(Associativity coherence)
	\\
	\\
	\begin{tikzcd}[font=\fontsize{9}{6}]
		FS^3 \arrow[rr,"{\varepsilon}_{S^2}"]
		\arrow[rrd,"F{\mu}_S"] \arrow[dd,"FS\mu"']
		&& FS^2 \arrow[rrd,"{\varepsilon}_S"] \arrow[d,Rightarrow,"m_S",shorten=1mm]
		\\
		&& FS^2 \arrow[rr,"{\epsilon}_S"'] \arrow[lld,"F\alpha",Rightarrow,shorten=7mm] \arrow[dd,"F\mu"']
		&& FS \arrow[lldd,"m"',Rightarrow,shorten=7mm]
		\arrow[dd,"\varepsilon"]
		\\
		FS^2 \arrow[drr,"F\mu"']
		\\
		&& FS \arrow[rr,"\varepsilon"']
		&& F
	\end{tikzcd} = \begin{tikzcd}[font=\fontsize{9}{6}]
		FS^3 \arrow[rr,"{\varepsilon}_{S^2}"] \arrow[dd,"FS\mu"']
		&& FS^2 \arrow[rrd,"{\varepsilon}_S"] \arrow[dd,"F\mu"'] \arrow[lldd,"{\varepsilon}_\mu",Rightarrow,shorten=7mm] \arrow[rrd,"{\varepsilon}_S"]
		\\
		&&&& FS \arrow[lld,"m"',Rightarrow,shorten=7mm]
		\arrow[dd,"\varepsilon"]
		\\
		FS^2 \arrow[drr,"F\mu"'] \arrow[rr,"{\varepsilon}_S"]
		&& FS \arrow[rrd,"\varepsilon"] \arrow[d,Rightarrow,"m"',shorten=1mm]
		\\
		&& FS \arrow[rr,"\varepsilon"']
		&& F
	\end{tikzcd}

\begin{remark}\label{Remark derivable law for pseudoalgebras}
	By Lemma 9.1 of \cite{Doctrines fully faithful adjoint string}, there is also the following derivable left unit coherence for the pseudoalgebra $\left(F, \varepsilon, u, m\right)$.
	
	$$\begin{tikzcd}[font=\fontsize{9}{6}]
		FS \arrow[rr,"FS\eta"{name=A}] \arrow[rrdd,"\text{id}"'{name=B}]
		&& FS^2 \arrow[dd,"F_{\mu}"] \arrow[rr,"{\varepsilon}_S"]
		&& FS \arrow[dd,"\varepsilon"] \arrow[lldd,Rightarrow,"m",shorten=5mm]
		\\
		\\
		&& FS \arrow[rr,"\varepsilon"']
		&& F
		\arrow[from=A,to=B,Rightarrow,"F\rho",shorten=2mm,shift left=2]
	\end{tikzcd} = \begin{tikzcd}
		FS \arrow[rr,"FS\eta"] \arrow[rrdd,"\varepsilon"']
		&& FS^2 \arrow[rr,"{\epsilon}_{S}"] \arrow[dd, Rightarrow,"{\varepsilon}_{\eta}"',shorten=4mm]
		&& FS \arrow[dd,"\varepsilon"']
		\\
		\\
		&& F \arrow[rruu,"F\eta"{name=A}] \arrow[rr,"\text{id}"'{name=B}]
		&& F \arrow[from=A,to=B,Rightarrow,shorten=2mm,shift left=3,"u"']
	\end{tikzcd}$$
	
\end{remark}

\begin{remark}\label{Gray Kleisli generators and relations}
	The description of a universal cocone determining $\mathcal{A}_S$ from Remark \ref{Gray Kleilsi's universal cocone as a weighted colimit} gives a presentation for $\mathcal{A}_S$ in terms of generators and relations. Specifically, $\mathcal{A}_S$ is constructed by
	
	\begin{itemize}
		\item Including all data from $\mathcal{A}$ subject to composition conditions in $\mathcal{A}$,
		\item Freely adding a generating $1$-cell $\varepsilon_{X}: SX \rightarrow X$ for every $X \in \mathcal{A}$,
		\item Adding the following invertible $2$-cells 
		
		$$\begin{tikzcd}[font=\fontsize{9}{6}]
			SX \arrow[rr, "Sf"]
			\arrow[dd,"{\varepsilon}_X"']
			&& SY \arrow[dd,"{\varepsilon}_Y"]
			\\
			\\
			X \arrow[rr,"f"'] \arrow[rruu,Rightarrow,shorten=6mm,"{\varepsilon}_f"]
			&& Y&{}
		\end{tikzcd} \begin{tikzcd}[font=\fontsize{9}{6}]
			X \arrow[rr,"{\eta}_X"] \arrow[rrdd,"1_X"'{name=A},bend right]
			&& SX \arrow[dd,"{\varepsilon}_X"] \arrow[to=A,Rightarrow,"{u}_X",shorten=5mm]
			\\
			\\
			&& X&{}
		\end{tikzcd} \begin{tikzcd}[font=\fontsize{9}{6}]
			{S^2}X \arrow[rr,"{\epsilon}_{SX}"] \arrow[dd,"{\mu}_X"']
			&& SX \arrow[dd,"{\epsilon}_X"] \arrow[lldd,Rightarrow,shorten=5mm,"m_X"]
			\\
			\\
			SX \arrow[rr,"{\varepsilon}_X"']
			&& X
		\end{tikzcd}$$
		\item Specifying the evident relations to ensure that $\left(X, f\right)\mapsto \left(\varepsilon_{X}, \varepsilon_{f}\right)$ is a pseudonatural transformation, that $X \mapsto u_{X}$ and $X \mapsto m_{X}$ are invertible modifications, and that the left unit and associativity coherences hold. Note that each of these conditions are indeed equations between parallel pairs of $2$-cells in $\mathcal{A}_S$, and that the modification conditions for $u$ and $m$ on a morphism $f: X \rightarrow Y$ respectively say that the pair $\left(f, \varepsilon_{f}\right)$ satisfies the unit and multiplication axioms for a pseudomorphism.
	\end{itemize}
	
	\noindent The remaining data of the $\mathbf{Gray}$-enriched Kleisli pseudoadjunction can be described as follows.
	
	\begin{itemize}
		\item The $\mathbf{Gray}$-enriched Kleisli left pseudoadjoint $F: \mathcal{A} \rightarrow \mathcal{A}_S$ is the evident $2$-functor which includes data from $\mathcal{A}$. 
		\item The right pseudoadjoint $U$ sends any data in $\mathcal{A}_S$ from $\mathcal{A}$ to its image under $S$, sends $\varepsilon_{X}$ to $\mu_{X}$, $\varepsilon_{f}$ to $\mu_{f}$, $u_{X}$ to $\rho_{X}$ and $m_{X}$ to $\alpha_{X}$.
		\item The counit has component at $X$ given by $\varepsilon_{X}$, constraint at $f$ from $\mathcal{A}$ given by $\varepsilon_{f}$ and constraint at $\varepsilon_{X}$ given by $m_{X}$.
		\item The left triangulator has component at $X$ given by $u_{X}$.
	\end{itemize}  
	
\end{remark}

\section{The canonical comparison $2$-functor}\label{Section the canonical comparison induced by a pseudoadjunction splitting the pseudomonad}

\noindent The aim of this section is to show that the comparison $2$-functor from $\mathcal{A}_S$ induced by a pseudoadjunction giving rise to $\left(\mathcal{A}, S\right)$ is bi-fully faithful. We begin by describing this comparison in Remark \ref{comparison from gray kleisli}. 

\subsection{Comparison to pseudoalgebras}\label{Subsection comparison to pseudoalgebras}

\begin{remark}\label{comparison from gray kleisli}
	Let $\left(\bar{F}: \mathcal{A} \rightarrow \mathcal{B}, \bar{U}: \mathcal{B} \rightarrow \mathcal{A}, \eta, \bar{\varepsilon}, \lambda, \bar{u}: 1_{\bar{F}} \Rrightarrow \bar{\varepsilon}_{\bar{F}}.\bar{F}\eta\right)$ be a pseudoadjunction with induced pseudomonad $\left(\mathcal{A}, S\right)$. Then the canonical comparison $2$-functor $K:\mathcal{A}_S \rightarrow \mathcal{B}$ sends 
	\begin{itemize}
		\item Any data in $\mathcal{A}_S$ from $\mathcal{A}$ to its image under $\bar{F}: \mathcal{A} \rightarrow \mathcal{B}$,
		\item $\varepsilon_{X}$ to $\bar{\varepsilon}_{\bar{F}X}$
		\item The generating $2$-cells $\varepsilon_{f}$, $u_{X}$, and $m_{X}$ to $\bar{\varepsilon}_{\bar{F}f}$, $ \bar{u}_X$, and $ \bar{\varepsilon}_{\bar{\varepsilon}_{\bar{F}X}}$ respectively.
	\end{itemize}
\end{remark}

\noindent Before proving bi-fully faithfulness of $K$ we describe the concrete example of when the pseudoadjunction is Eilenberg-Moore.

\begin{example}\label{comparison to free pseudoalgebras}
	When the pseudoadjunction of Remark \ref{comparison from gray kleisli} is the Eilenberg-Moore pseudoadjunction then $K$ sends
	
	\begin{itemize}
		\item $X$ to $\left(SX, \mu_X, \lambda_{X}, \alpha_{X}\right)$, which will be abbreviated as $\left(SX, \mu_{X}\right)$,
		\item $f: X \rightarrow Y$ to $\left(f, \mu_{f}\right)$ whenever $f$ is a morphism in $\mathcal{A}$, and similarly for generating $2$-cells from $\mathcal{A}$.
		\item $\varepsilon_{X}$ to $\left(\mu_{X}, \alpha_{X}\right): \left(S^2X, \mu_{S^2 X}\right) \rightarrow \left(SX, \mu_{X}\right)$,
		\item The generating $2$-cells from $\mathcal{A}$ to their images under $S$, and the $2$-cells $\varepsilon_{f}$, $u_{X}$ and $m_{X}$ to the following $2$-cells.
	\end{itemize}
	$$\begin{tikzcd}[font=\fontsize{9}{6}]
		\left({S^2}X{,}{\mu}_{SX}\right) \arrow[rr, "\left({S^2}f{,}{\mu}_{Sf}\right)"]
		\arrow[dd,"\left({\mu}_X{,}{\alpha}_X\right)"' description]
		&& \left({S^2}Y{,}{\mu}_{SY}\right) \arrow[dd,"\left({\mu}_Y{,}{\alpha}_Y\right)" description]
		\\
		\\
		\left(SX{,}{\mu}_X\right) \arrow[rr,"\left(Sf{,}{\mu}_{f}\right)"'] \arrow[rruu,Leftarrow,shorten=6mm,"{\mu}_f"]
		&& \left(SY{,}{\mu}_Y\right)
	\end{tikzcd} 
	\begin{tikzcd}[font=\fontsize{9}{6}]
		\left(SX{,}{\mu}_X\right) \arrow[rr,"\left(S{\eta}_{X}{,}{\mu}_{{\eta}_{X}}\right)"] \arrow[rrdd,"{\text{id}}_{\left(SX{,}{\mu}_X\right)}"'{name=A},bend right]
		&& \left({S^2}X{,}{\mu}_{SX}\right) \arrow[dd,"\left({\mu}_{X}{,}{\alpha}_{X}\right)" description] \arrow[to=A,Rightarrow,"{\rho}_X",shorten=5mm]
		\\
		\\
		&& \left(SX{,}{\mu}_X\right)
	\end{tikzcd}\begin{tikzcd}[font=\fontsize{9}{6}]
		\left({S^3}X{,}{\mu}_{{S^2}X}\right) \arrow[rr,"\left({\mu}_{SX}{,}{\alpha}_{SX}\right)"] \arrow[dd,"\left(S{\mu}_{X}{,}{\mu}_{{\mu}_{X}}\right)"' description]
		&& \left({S^2}X{,}{\mu}_{{S}X}\right) \arrow[dd,"\left({\mu}_{X}{,}{\alpha}_{X}\right)" description] \arrow[lldd,Rightarrow,shorten=5mm,"{\alpha}_X"]
		\\
		\\
		\left({S^2}X{,}{\mu}_{{S}X}\right) \arrow[rr,"\left({\mu}_{X}{,}{\alpha}_{X}\right)"']
		&& \left({S}X{,}{\mu}_{X}\right)
	\end{tikzcd}$$ 		
\end{example}

\begin{notation}\label{Notation free pseudoalgebras}
	Note that the $2$-functor $K$ of Example \ref{comparison to free pseudoalgebras} factorises through the full sub-$2$-category of $\mathcal{A}^S$ on those pseudoalgebras which are free on an object of $\mathcal{A}$. We will denote this $2$-category by $\mathbf{FreePsAlg}\left(\mathcal{A}, S\right)$.
\end{notation}

\subsection{Bi-fully-faithfulness of the comparison}\label{Subsection bi-fully-faithfulness} We now prove that all comparison $2$-functors of the form $K: \mathcal{A}_S \rightarrow \mathcal{B}$ described in Remark \ref{comparison from gray kleisli} are bi-fully faithful, which is to say that their actions on hom-categories are equivalences. 

\begin{theorem}\label{comparison from Gray Kleisli is bi-fully-faithful}
	The $2$-functor $K: \mathcal{A}_S \rightarrow \mathcal{B}$ of Remark \ref{comparison from gray kleisli} is bi-fully faithful.
\end{theorem}

\begin{proof}
	Consider the following pasting diagram in $\mathbf{Cat}$.
	
	$$\begin{tikzcd}[column sep = 15, font=\fontsize{9}{6}]
		\mathcal{B}\left({FX}{,}{FY}\right)
		\arrow[rr,"U", blue]
		\arrow[rdddd,"\mathcal{B}\left({\varepsilon_{FX}}{,}{1}\right)"]
		\arrow[dddddd,bend right = 30, "1"']
		&
		{}
		\arrow[dddd,Rightarrow,shorten=40,shift left=8,"\varepsilon"]
		&
		\mathcal{A}\left({UFX}{,}{UFY}\right)
		\arrow[rr,"\mathcal{A}\left({\eta_{X}}{,}{1}\right)", blue]
		\arrow[dd,"F"]
		&&
		\mathcal{A}\left({X}{,}{UFY}\right)
		\arrow[dddddd,"F_S", red]
		\arrow[ldddd,"F"']
		\\
		&&&
		=
		\\
		&
		{} \arrow[dd,shift right=30,Rightarrow,shorten=8,"\mathcal{B}\left({\lambda_X}{,}{1}\right)"]
		&
		\mathcal{B}\left({FUFX}{,}{FUFY}\right)
		\arrow[ldd,"\mathcal{B}\left(1{,}{\varepsilon_{FY}}\right)"']
		\arrow[rdd,"\mathcal{B}\left(F{\eta_{X}}{,}{1}\right)"]
		\\
		\\
		&
		\mathcal{B}\left({FUFX}{,}{FY}\right)
		\arrow[ldd,"\mathcal{B}\left(F{\eta_{X}}{,}{1}\right)" description]
		&
		=
		&
		\mathcal{B}\left({FX}{,}{FUFY}\right)
		\arrow[llldd,"\mathcal{B}\left(1{,}{\varepsilon_{FY}}\right)" near start]
		\arrow[r,Rightarrow,no head,shorten=29.3]
		\arrow[dd,Rightarrow,shorten=10,"\varepsilon"']
		&
		{}
		\\
		\\
		\mathcal{B}\left({FX}{,}{FY}\right)
		&&
		\mathcal{A}_{S}\left({{F_S}X}{,}{{F_S}Y}\right)
		\arrow[ll,"K"]
		&
		{}
		&
		\mathcal{A}_{S}\left({{F_S}X}{,}{{F_S}UFY}\right)
		\arrow[ll,"\mathcal{A}_{S}\left({1}{,}{{\varepsilon}_{FY}}\right)", red]
		\arrow[luu,"K"]
	\end{tikzcd}$$
	
	\noindent Observe that all $2$-cells appearing in this pasting are invertible, so by the two-out-of-three property of equivalences of categories it suffices to show that both of the following statements are true.
	\begin{enumerate}
		\item The composite of the functors depicted in \color{blue}blue\color{black} \hspace{1mm} are an equivalence.
		\item The composite of the functors depicted in \color{red}red\color{black} \hspace{1mm} are an equivalence.
	\end{enumerate}
	\noindent But (1) holds by the pseudoadjunction $F \dashv U$ and (2) holds by the pseudoadjunction $\bar{F} \dashv \bar{U}$. 
\end{proof}

\noindent The following result follows easily from Theorem \ref{comparison from Gray Kleisli is bi-fully-faithful}. It can be seen as categorifying the well-known analogous statement for monads, with biessential surjectivity on objects replacing surjectivity on objects, and weak biequivalences replacing equivalences of categories.

\begin{theorem}\label{when is comparison from gray kleisli a biequivalence}
	For the pseudoadjunction described in Remark \ref{comparison from gray kleisli}, the following conditions are equivalent \begin{enumerate}
		\item The comparison $2$-functor $K: \mathcal{A}_S \rightarrow \mathcal{B}$ is a weak biequivalence.
		\item $\bar{F}$ is biessentially surjective on objects.
		\item For every $2$-category $\mathcal{C}$, the canonical comparison depicted below is a weak biequivalence.
	\end{enumerate}
	$$\begin{tikzcd}
		\mathbf{Gray}\left(\mathcal{B}{,} \mathcal{C}\right) \arrow[rr, "\mathbf{Gray}\left(K{,} \mathcal{C}\right)"] && \mathbf{Gray}\left(\mathcal{A}_{S}{,} \mathcal{C}\right) \arrow[rr, "\cong"] && {\mathbf{Gray}\left(\mathcal{A}{,} C\right)}^{\mathbf{Gray}\left(S{,} \mathcal{C}\right)}
	\end{tikzcd} $$
\end{theorem}

\begin{proof}
	Note that $\bar{F} = KF$ with $F$ biessentially surjective on objects. Thus $(a) \iff (b)$ follows from the cancellation property of biessentially surjective on objects functors. The equivalence $(a) \iff (c)$ follows by observing that, since all $2$-categories are fibrant, the representables $\mathbf{Gray}\left(-, \mathcal{C}\right)$ preserve and jointly reflect weak equivalences in the Lack model structure of \cite{Quillen 2-cat}, and these are precisely weak biequivalences.
\end{proof}

\subsection{Applications}\label{Subsection Applications} A simpler description of the $2$-cells in $\mathcal{A}_S$ is given in Corollary \ref{2-cells in A_S are 2-cells of free pseudoalgebras}, to follow.

\begin{corollary}\label{2-cells in A_S are 2-cells of free pseudoalgebras}
	Let $f, g \in \mathcal{A}_{S}\left(X, Y\right)$ be arbitrary morphisms. The $2$-cells $\phi: f \Rightarrow g$ in $\mathcal{A}_{S}$ may equivalently be described as $2$-cells between the pseudomorphisms of free pseudoalgebras corresponding to $Kf$ and $Kg$.
\end{corollary}

\begin{proof}
	Apply Theorem \ref{comparison from Gray Kleisli is bi-fully-faithful} to the free pseudoalgebras pseudoadjunction described in Example \ref{comparison to free pseudoalgebras}. The result follows from the local fully-faithfulness aspect of $K$ being bi-fully faithful.
\end{proof}

\noindent Since the left pseudoadjoint $\mathcal{A} \rightarrow \mathbf{FreePsAlg}\left(\mathcal{A}, S\right)$ is biessentially surjective on objects it follows that $\mathcal{A}_S \rightarrow \mathbf{FreePsAlg}\left(\mathcal{A}, S\right)$ is a weak biequivalence. We end this section by showing that this comparison is only an isomorphism when $\mathcal{A} = \mathbf{0}$, and noting that $\mathbf{Gray}$-Kleisli left pseudoadjoints are not closed under composition.

\begin{proposition}\label{Gray Kleilsi of identity pseudomonads}
	Let $\mathcal{A}$ be a $2$-category and $S$ the identity pseudomonad on $\mathcal{A}$. Denote by $\pi_{2}\left(\mathbf{Mnd}\right)$ the $2$-category constructed from the free-living monad $\mathbf{Mnd}$ by universally inverting the $2$-cells. Then
	\begin{enumerate}
		\item $\mathcal{A}_{S}$ is isomorphic to the Gray tensor product $\pi_{2}\left(\mathbf{Mnd}\right) \otimes \mathcal{A}$. 
		\item The biequivalence $\mathcal{A}_S \rightarrow \mathbf{FreePsAlg}\left(\mathcal{A}, S\right)$ is an isomorphism of $2$-categories if and only if $\mathcal{A} = \mathbf{0}$. 
	\end{enumerate} 
\end{proposition}

\begin{proof}
	The morphisms in $\mathcal{A}_S$ are freely generated by those in $\mathcal{A}$ and a morphism $\varepsilon_{X}: X \rightarrow X$, hence morphisms $\varepsilon_{X, n}: X \rightarrow X$ subject to $\varepsilon_{X,n}\circ \varepsilon_{X, m} = \varepsilon_{X, n+m}$ for natural numbers $m, n$ with $\varepsilon_{X, 0} = 1_{X}$. This is the same as in the Gray tensor product. For $2$-cells, the generators $m_{X}:\varepsilon_{X, 1}.\varepsilon_{X, 1} =\varepsilon_{X, 2} \Rightarrow \varepsilon_{X}$ and $u_{X}: 1_{X}= \varepsilon_{X, 0} \Rightarrow \varepsilon_{X, 1}$ become the generating $2$-cells in $\mathbf{Mnd}$ and the pseudoalgebra relations in $\mathcal{A}_S$ amount to saying that $u$ and $m$ are invertible and satisfy the monoid axioms. Finally, the generators $\varepsilon_{f}:\varepsilon_{X}.f \Rightarrow f.\varepsilon_{X}$ give rise to the generators $\varepsilon_{X, 1}.f \Rightarrow f\varepsilon_{X, 1}$ in the Gray tensor product, and pseudonaturality of $\varepsilon$ says that this satisfies the relations in the Gray tensor product.
	\\
	\\
	\noindent For part (2), observe that if $\mathcal{A}$ has an object $X$ then the morphism $\varepsilon_{X}: SX \rightarrow X$ is freely added in $\mathcal{A}_{S}$, and in particular is not in the image of $F_S$ for any $SX \rightarrow X$ in $\mathcal{A}$. In particular, $F_S$ is not full on $1$-cells. Now by the cancellation property of $2$-functors which are full on morphisms, since $\mathcal{A} \rightarrow \mathbf{FreePsAlg}\left(\mathcal{A}, S\right)$ is such but $F_S: \mathcal{A} \rightarrow \mathcal{A}_{S}$ is not, $K: \mathcal{A}_{S} \rightarrow \mathbf{FreePsAlg}\left(\mathcal{A}, S\right)$ can not be full on morphisms. In particular it is not an isomorphism if $\mathcal{A}$ has any objects.
\end{proof}

\begin{remark}
	Note that if the invertibility relations on $u$ and $m$ were omitted, such as would be the case if the data $\left(F_{S}, \varepsilon, u, m\right)$ were only required to be a lax algebra for $\left(-\right)\circ S$, then the resulting $2$-category would be isomorphic to $\mathbf{Mnd} \otimes \mathcal{A}$, a similar construction but without freely inverting the $2$-cells of $\mathbf{Mnd}$.
\end{remark}

\begin{corollary}\label{Gray Kleisli not closed under composition}
	$\mathbf{Gray}$-enriched Kleisli left pseudoadjoints are not closed under composition. 
\end{corollary}

\begin{proof}
	Consider the $2$-functor $\mathbf{1} \rightarrow \pi_{2}\left(\mathbf{Mnd}\right)$, which is a special case of Proposition \ref{Gray Kleilsi of identity pseudomonads}. Then apply the same construction again for another $\mathbf{Gray}$-enriched Kleisli left pseudoadjoint $\pi_{2}\left(\mathbf{Mnd}\right) \rightarrow \pi_{2}\left(\mathbf{Mnd}\right) \otimes \pi_{2}\left(\mathbf{Mnd}\right)$. Since this second $2$-functor is not an isomorphism, and since there is only one pseudomonad on $\mathbf{1}$, the composite is not itself a $\mathbf{Gray}$-enriched Kleisli left pseudoadjoint.
\end{proof}

\section{Homotopy theoretic aspects}\label{Section Homotopy Theoretic Aspects}

\noindent In Section \ref{Section related notions and perspectives} we recalled the Kleisli constructions for pseudomonads which seem to be of greatest mathematical interest, namely the Kleisli bicategory and the $2$-category of free pseudoalgebras. In contrast, in Subsection \ref{Subsection the presentation of the Kleisli object of a pseudomonad} we were only able to describe $\mathbf{Gray}$-enriched Kleisli objects for pseudomonads via a presentation. As per Theorem \ref{when is comparison from gray kleisli a biequivalence}, left pseudoadjoints which only satisfy the universal property up to biequivalence have been characterised as those which are biessentially surjective on objects, a property that is closed under composition. In contrast, Corollary \ref{Gray Kleisli not closed under composition} recorded that those left pseudoadjoints whose codomains satisfy the enriched universal property are not closed under composition.
\\
\\
\noindent The discrepancies described in the previous paragraph are unfortunate both for applications and for the development of the theory of pseudomonads. Closure under composition is an essential ingredient used in \cite{FTM2} to analyse free cocompletions under Kleisli objects for monads, and to develop the theory of wreaths. Any attempt to extend this theory to the two-dimensional setting would need to address these issues. In this Section we prove some useful homotopical facts about the weight for $\mathbf{Gray}$-enriched Kleisli objects and the $\mathbf{Gray}$-enriched Kleisli left pseudoadjoint. In the forthcoming \cite{Miranda Tricategorical Limits and Colimits} we will use these properties to characterise the left pseudoadjoints in $\mathbf{Gray}$ that satisfy a tricategorical universal property as precisely those which are biessentially surjective on objects.

\subsection{Preliminaries}\label{Subsection Homotopical preliminaries}

\begin{notation}
	Let $\left(\mathcal{V}, \otimes, I\right)$ be a monoidal category with an initial object $\mathbf{0}$. For $X$ an object of $\mathcal{V}$, there is a $\mathcal{V}$-category $\mathbf{2}_\mathcal{V}[X]$ with two objects $0, 1$, and homs given by $\mathbf{2}_\mathcal{V}[X]\left(0, 1\right)=X$, $\mathbf{2}_\mathcal{V}[X]\left(0, 0\right)= \mathbf{2}_\mathcal{V}[X]\left(1, 1\right) = I$, and $\mathbf{2}_\mathcal{V}[X]\left(1, 0\right)= \mathbf{0}$. In this section, $\mathcal{V}$ will be either $\mathbf{Set}$ or $\mathbf{Cat}$ with the cartesian monoidal structure, and the subscript $\mathcal{V}$ will be omitted when it is clear from context.
\end{notation}

\begin{notation}
	In a model category \cite{Categorical Homotopy Theory}, \cite{Hovey Model Categories}
	\begin{itemize}
		\item fibrations will be denoted \begin{tikzcd}
			\bullet \arrow[r, two heads] & \bullet
		\end{tikzcd}
		\item weak equivalences will be denoted \begin{tikzcd}
			\bullet \arrow[r, "\sim"] & \bullet
		\end{tikzcd}
		\item cofibrations will be denoted \begin{tikzcd}
			\bullet \arrow[r, rightarrowtail] & \bullet
		\end{tikzcd}
	\end{itemize}
\end{notation}

\noindent The model structure of relevance is the Lack model structure on $2$-$\mathbf{Cat}$, which we now review.

\begin{example}\label{Example Lack model structure}
	In \cite{Quillen 2-cat}, a model structure is defined on the category $2$-$\mathbf{Cat}$ of $2$-categories and $2$-functors. Its weak equivalences are the $2$-functors which are biessentially surjective on objects and equivalences of categories between homs, while its fibrations are the so called \emph{equiv-fibrations}. Equiv-fibrations are those $2$-functors which are isofibrations between hom-categories, and which also lift adjoint equivalences. In this model structure, all objects are fibrant while the cofibrant objects are precisely those $2$-categories whose underlying categories are free on graphs. This model structure is combinatorial, and is monoidal with respect to the $\mathbf{Gray}$-tensor product. The generating cofibrations in this model structure are depicted below. Note that apart from $i_{1}$, the generating cofibrations are all of the form $\mathbf{2}[f]$ for some $f$ a generating cofibration in the canonical model structure on $\mathbf{Cat}$.

$$
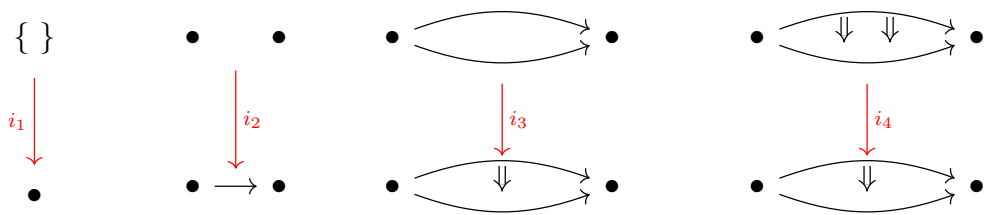
\begin{tikzcd}
	\{ \text{ } \}
	\arrow[dd,"i_1"',shorten=4.5,red]
	\\
	\\
	\bullet
	& {}
\end{tikzcd} \begin{tikzcd}[column sep=3]
	\bullet
	&
	{}
	\arrow[dd,"i_2",shorten=5,red]
	&
	\bullet
	\\
	\\
	\bullet
	\arrow[rr]
	&
	{}
	&
	\bullet
	& {}
	& {}
\end{tikzcd} \begin{tikzcd}
	\bullet
	\arrow[rr,bend left=20]
	\arrow[rr,bend right=20]
	&
	{}
	\arrow[dd,"i_3",shorten=10,red]
	&
	\bullet
	\\
	\\
	\bullet
	\arrow[rr,bend left=20,""{name=A}]
	\arrow[rr,bend right=20,""{name=B}]
	&
	{}
	&
	\bullet
	\arrow[from=A,to=B,Rightarrow,shorten=2]
	&
	{}
\end{tikzcd} \begin{tikzcd}
	\bullet
	\arrow[rr,bend left=20,""{name=C}]
	\arrow[rr,bend right=20,""{name=D}]
	&
	{}
	\arrow[dd,"i_4",shorten=10,red]
	&
	\bullet
	\\
	\\
	\bullet
	\arrow[rr,bend left=20,""{name=A}]
	\arrow[rr,bend right=20,""{name=B}]
	&
	{}
	&
	\bullet
	\arrow[from=A,to=B,Rightarrow,shorten=2]
	\arrow[from=C,to=D,Rightarrow,shift left=3,shorten=2]
	\arrow[from=C,to=D,Rightarrow,shift right=3,shorten=2]
\end{tikzcd}$$

\end{example}

\noindent The following proposition follows immediately from the generators and relations description of $\mathbf{Gray}$-enriched Kleisli objects $\mathcal{A}_S$ given in Remark \ref{Gray Kleisli generators and relations}. 

\begin{proposition}\label{gray kleisli left adjoint a cofibration}
	The left pseudoadjoint $F: \mathcal{A} \rightarrowtail \mathcal{A}_S$ is a cofibration in the Lack model structure on $2$-$\mathbf{Cat}$.
\end{proposition}

\begin{proof}
	As described in Remark \ref{Gray Kleisli generators and relations}, the $1$-cells of $\mathcal{A}_S$ are constructed by freely adding certain $1$-cells to $\mathcal{A}$ via $F$. Thus $F$ is a cofibration by Lemma 4.1 and Corollary 4.1.2 of \cite{Quillen 2-cat}.
\end{proof}

\begin{definition}\label{definition projective cofibrations}
	Let $\mathcal{V}$ be a combinatorial monoidal model category with set of generating cofibrations $\mathcal{G}$ and let $\mathcal{C}$ be enriched in $\mathcal{V}$. Consider the projective model structure on the $\mathcal{V}$-enriched functor category $[\mathcal{C}^\text{op}, \mathcal{V}]$, whose fibrations and weak equivalences are defined pointwise and whose set of generating cofibrations in $[\mathcal{C}^\text{op}, \mathcal{V}]$ is given as depicted below, where $\cdot$ denotes copower.
	
	$$\{f\cdot \mathcal{C}\left(-{,} X\right): A \cdot \mathcal{C}\left(-{,}X\right) \rightarrow  B \cdot \mathcal{C}\left(-{,}X\right) |\left(f: A \rightarrow B\right) \in \mathcal{G}{,} X \in \mathcal{C}\}$$
	
	\noindent A projective cofibration in $[\mathcal{C}^\text{op}, \mathcal{V}]$ is called
	\begin{itemize}
		\item a \emph{connected cellular} if it can be built out of generating cofibrations using just pushouts of spans $X \leftarrow Y \rightarrow Z$ where $Y \neq \mathbf{0}$, and countable composites.
		\item a \emph{cellular} if it can be built out of generating cofibrations using just arbitrary pushouts, and countable composites.
	\end{itemize} 
	\noindent A projective cofibrant object $W \in [\mathcal{C}^\text{op}, \mathcal{V}] $ will be called cellular (resp. connected cellular) if the projective cofibration $\mathbf{0} \rightarrowtail W$ is cellular (resp. projective cellular).
\end{definition}

\begin{remark}\label{general projective cofibrants, and relation to flexibility in Cat}
	Retracts are also needed to build a general projective cofibration. Projective cofibrant weights on finite dimensional categories are homotopically well-behaved in that when taking limits equations only need to be imposed on cells of the highest dimension, while when taking colimits one only needs to quotient by equivalence relations on cells of the highest dimension. We review these notions when $\mathcal{V}= \mathbf{Cat}$. 
	\begin{itemize}
		\item The connected cellular projective cofibrant objects are the saturation of the weights for (co)inserters and (co)equifiers.
		\item The cellular projective cofibrant objects are the saturation of the weights for (co)products, (co)inserters and (co)equifiers. These are often referred to by the acronym \emph{PIE}. They are coalgebras for the cofibrant replacement comonad, which on $[\mathcal{C}^\text{op}, \mathbf{Cat}]$ is equivalently described as the $2$-comonad whose left $2$-adjoint is the pseudonatural transformation classifier $L: [\mathcal{C}^\text{op}, \mathbf{Cat}] \rightarrow \mathbf{GRAY}\left(\mathcal{C}^\text{op}, \mathbf{Cat}\right)$.
		\item The projective cofibrant objects are the saturation of the weights for (co)products, (co)inserters, (co)equifiers, and pseudo splittings of pseudo-idempotents. These are often called \emph{flexible}. They are equivalently those $W$ for which the counit $\varepsilon_{W}: LW \rightarrow W$ $2$-natural transformation has a section that is also strict, as opposed to merely pseudonatural.
	\end{itemize}
	
	\noindent See \cite{Categorical Homotopy Theory}, \cite{Shulman Enriched Homotopy Limits and Colimits} and \cite{Moser Injective and Projective Mondel Structures on Enriched Diagram Categories} for the projective model structure. See \cite{Power Robinson PIE} for PIE weights, \cite{Bird Kelly Power Street} for flexible weights, \cite{Two Dimensional Monad Theory} and \cite{Lack Homotopy Theoretic Aspects of 2-monads} for a $2$-monadic perspective on flexibility and its relation to cofibrancy, and \cite{Albert Kelly closure of a class of colimits} for saturations of classes of weights.
\end{remark}

\subsection{Projective cofibrancy of the weight for Kleisli objects}\label{Subsection Kleilsi weight is projective cofibrant}
\noindent The main result of this section is Corollary \ref{TriKL weight is cofibrant}, in which we construct the weight for Kleisli objects showing that it is projective cellular cofibrant. This construction is dual to what will be described in the proof of Proposition \ref{Pseudoalgebras via cofibrant weighted limits}, which will build the Eilenberg-Moore object of a pseudomonad using $2$-cell inserters, iso-$3$-cell inserters, and $3$-cell equifiers.
\\
\\
\noindent Recall that in $2$-category theory, the weight for iso-inserters is projective cofibrant since iso-inserters can be constructed using inserters and equifiers. In a similar way, iso-$3$-cell inserters in a $\mathbf{Gray}$-category can be constructed using $3$-cell inserters and $3$-cell equifiers. To see this we need Lemma \ref{iso 3 cell inserter from 3 cell inserter and 3 cell equifier}, whose proof resembles the proof of the analogous result about iso-inserters in dimension two. The proof is given in Appendix \ref{Proof of Lemma 3-cell inserter}.

\begin{lemma}\label{iso 3 cell inserter from 3 cell inserter and 3 cell equifier}
	The $2$-functor $k: \mathbf{2}[\mathbf{1}+\mathbf{1}] \rightarrow \mathbf{2}[\mathbb{I}]$ depicted below can be built as a finite composite of pushouts of copowers of the generating cofibrations in the Lack model structure recalled in Example \ref{Example Lack model structure}. In particular, it is itself a cofibration.
	
	$$\begin{tikzcd}
		X\arrow[rr, bend left = 30, "f"name=A]\arrow[rr, bend right = 30, "g"'name=B] &{}& Y
		&{}\arrow[rrr, mapsto]&&&{}
		&X\arrow[rr, bend left = 30, "f"name=C]\arrow[rr, bend right = 30, "g"'name=D] &\cong_{\alpha}& Y
	\end{tikzcd}$$
	
\end{lemma}

\noindent We now review the limits out of which Eilenberg-Moore objects of pseudomonads, described explicitly in Subsection \ref{Subsection free pseudoalgebras}, can be built. Building them in this way exhibits their forgetful $2$-functors as equiv-fibrations.

\begin{definition}\label{cofibrantly weighted limits in Gray}	
	The $\mathbf{Gray}$-enriched power of $\mathcal{B}$ by $\mathcal{P}$ will be denoted by $\mathcal{B}^\mathcal{P}$, as shorthand for $\mathbf{Gray}\left(\mathcal{P}, \mathcal{B}\right)$. We define various notions of inserters and equifiers as pullbacks of powers in $\mathbf{Gray}$.
	\begin{enumerate}
		\item Let $F, G: \mathcal{A} \rightarrow \mathcal{B}$ be a parallel pair of $2$-functors. Then the \emph{$2$-cell inserter} of $F$ and $G$ is the pullback depicted below left, where the isomorphism $\mathcal{B}\times \mathcal{B} \cong \mathcal{B}^{\mathbf{1}+\mathbf{1}}$ is omitted.
		\item Let $F, G$ be as above and let $\alpha, \beta: F \Rightarrow G$ be a parallel pair of pseudonatural transformations. Then the \emph{$3$-cell inserter} of $\alpha$ and $\beta$ is the pullback depicted below right.
		$$\begin{tikzcd}[font=\fontsize{9}{6}]
			\mathbf{Ins}_{2}\left(F{,}G\right)
			\arrow[rr]
			\arrow[dd, two heads] && \mathcal{B}^\mathbf{2}\arrow[dd, "\mathcal{B}^{i_{2}}", two heads]
			\\
			\\
			\mathcal{A} \arrow[rr, "\left(F{,}G\right)"']&& \mathcal{B}\times \mathcal{B}&{}
		\end{tikzcd}\begin{tikzcd}[font=\fontsize{9}{6}]
			\mathbf{Ins}_{3}\left(\alpha{,}\beta\right)
			\arrow[rr]
			\arrow[dd, two heads] && {
				{\mathcal{B}}^{\mathbf{2}_\mathbf{Cat}{[}\mathbf{2}{]}}\arrow[dd, "{\mathcal{B}}^{i_{3}}", two heads]}
			\\
			\\
			\mathcal{A}\arrow[rr, "\left(\alpha{,}\beta\right)"']&& {
				\mathcal{B}}^{\mathbf{2}_{\mathbf{Cat}}{[}\mathbf{1}+\mathbf{1}{]}}
		\end{tikzcd}$$
		\item Let $F, G: \mathcal{A} \rightarrow \mathcal{B}, \alpha, \beta: F \Rightarrow G$ be as above. Then the \emph{iso-$3$-cell inserter} of $\alpha$ and $\beta$ is the pullback depicted below left, where $k$ is the $2$-functor of Lemma \ref{iso 3 cell inserter from 3 cell inserter and 3 cell equifier}.
		\item Let $\alpha$ and $\beta$ be as above and let $\Gamma, \Pi: \alpha \Rrightarrow \beta$ be a parallel pair of modifications. Then the \emph{$3$-cell equifier} of $\Gamma$ and $\Pi$ is the pullback depicted below right.

		$$\begin{tikzcd}[font=\fontsize{9}{6}]
			\mathbf{IsoIns}_{3}\left(\alpha{,}\beta\right)
			\arrow[rr]
			\arrow[dd, two heads]
			&& {{\mathcal{B}}^{\mathbf{2}_\mathbf{Cat}{[}\mathbb{I}{]}}\arrow[dd, "{\mathcal{B}}^{k}", two heads]}
			\\
			\\
			\mathcal{A}\arrow[rr, "\left(\alpha{,}\beta\right)"']
			&& {\mathcal{B}}^{\mathbf{2}_{\mathbf{Cat}}{[}\mathbf{1}+\mathbf{1}{]}}&{}
		\end{tikzcd}\begin{tikzcd}[font=\fontsize{9}{6}]
			\mathbf{Equif}_{3}\left(\Gamma{,}\Pi\right)
			\arrow[rr]
			\arrow[dd, two heads] && {
				{\mathcal{B}}^{\mathbf{2}_\mathbf{Cat}{[}\mathbf{2}_{\mathbf{Set}}{[}\mathbf{1}{]}{]}}\arrow[dd, "{\mathcal{B}}^{i_{4}}", two heads]}
			\\
			\\
			\mathcal{A}\arrow[rr, "\left(\Gamma{,}\Pi\right)"']&& {
				\mathcal{B}}^{\mathbf{2}_{\mathbf{Cat}}{[}\mathbf{2}_\mathbf{Set}{[}\mathbf{1}+\mathbf{1}{]}{]}}
		\end{tikzcd}$$
	\end{enumerate}
\end{definition}

\begin{remark}
	For each of the limits in Definition \ref{cofibrantly weighted limits in Gray}, its name is justified by the cone determined via the horizontally depicted projection. Note that an (resp. iso) $n$-cell inserter universally inserts an $n$-cell in $\mathbf{Gray}$. Thus when $n=2$ this is a pseudonatural transformation, while when $n=3$ this is a (resp. an invertible) modification. 
\end{remark}

\begin{proposition}\label{Pseudoalgebras via cofibrant weighted limits}
	The pseudomonadic $2$-functor $U^S: \mathcal{A}^S \rightarrow \mathcal{A}$ can be built in $\mathbf{Gray}$ as a finite composite pullbacks of powers of the generating cofibrations in the Lack model structure recalled in Example \ref{Example Lack model structure}. In particular, it is an equiv-fibration.
\end{proposition}

\begin{proof}
	We first construct the $2$-cell inserter $\mathbf{Ins}_{1}\left(S, 1_{\mathcal{A}}\right)$, whose universal cone is depicted below. This $2$-category has \begin{itemize}
		\item objects given by $\left(X, \phi_{X}\right)$ with $\phi: SX \rightarrow X$ a morphism in $\mathcal{A}$,
		\item arrows given by $\left(f, \bar{f}\right): \left(X, \phi_{X}\right) \rightarrow \left(Y, \phi_{Y}\right)$ where $\bar{f}: \phi_{Y}.Sf \Rightarrow f.\phi_{X}$ is invertible,
		\item $2$-cells $\beta: \left(f, \bar{f}\right) \Rightarrow \left(g, \bar{g}\right)$ given by $2$-cells $\beta: f \Rightarrow g$ in $\mathcal{A}$ satisfying the usual coherence for a $2$-cell of pseudoalgebras in $\mathcal{A}^S$.
	\end{itemize}
	
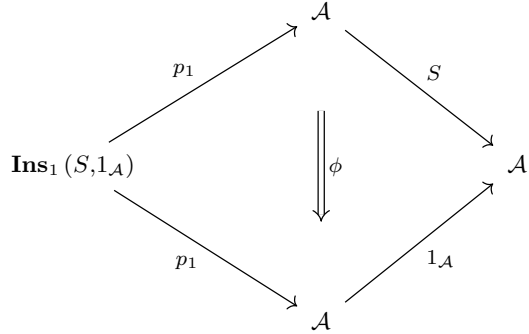
	$$\begin{tikzcd}[font=\fontsize{9}{6}]
		&&
		\mathcal{A}
		\arrow[dddd,Rightarrow,shorten=30,"\phi"]
		\arrow[rrdd,"S"]
		\\
		\\
		{\mathbf{Ins}_1\left(S{,}1_{\mathcal{A}}\right)}
		\arrow[rruu,"p_1"]
		\arrow[rrdd,"p_1"']
		&&&&
		\mathcal{A}&{}
		\\
		\\
		&&
		\mathcal{A}
		\arrow[rruu,"1_\mathcal{A}"']
	\end{tikzcd}$$
	
	\noindent Next, we will construct two successive iso-$3$-cell inserters from $\mathbf{Ins}_{1}\left(S, 1_{\mathcal{A}}\right)$. The first of these will have universal cone consisting of a projection $p_{2}: \mathbf{Ins}_{2}\left(\eta p_{1}.\phi , 1_{p_{1}}\right) \rightarrow \mathbf{Ins}_{1}\left(s, 1_{\mathcal{A}}\right)$, and an invertible modification $u$ as depicted below.
	
	$$
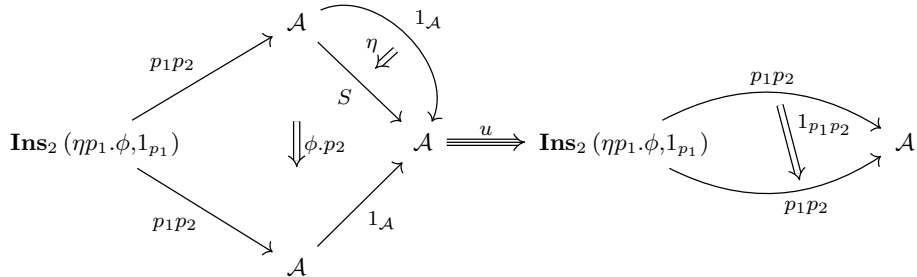
\begin{tikzcd}[row sep = 15, column sep = 15, font=\fontsize{9}{6}]
		&&
		\mathcal{A}
		\arrow[dddd,Rightarrow,shorten=30,"\phi.p_{2}"]
		\arrow[rrdd,"S"'{name=B}]
		\arrow[rrdd,bend left=70,"1_\mathcal{A}"{name=A}]
		\\
		\\
		\mathbf{Ins}_{2}\left(\eta p_{1}.\phi{,}1_{p_{1}}\right)
		\arrow[rruu,"p_1p_{2}"]
		\arrow[rrdd,"p_1p_{2}"']
		&&&&
		\mathcal{A}
		\arrow[rr,Rightarrow,"u"]
		\arrow[rr, no head]
		&&
		{\mathbf{Ins}_2}\left(\eta p_{1}.\phi{,}1_{p_{1}}\right)
		\arrow[rrrr,bend left,"p_1p_{2}"{name=C}]
		\arrow[rrrr,bend right,"p_1p_{2}"'{name=D}]
		&&&&
		\mathcal{A} &{}
		\\
		\\
		&&
		\mathcal{A}
		\arrow[rruu,"1_\mathcal{A}"']
		\arrow[from=A,to=B,Rightarrow,shorten=8,"\eta"']
		\arrow[from=C,to=D,Rightarrow,shorten=5,"1_{p_1p_{2}}"]
	\end{tikzcd}$$
	
	\noindent The $2$-category $\mathbf{IsoIns}_{2}\left(\eta p_{1}.\phi{,}1_{p_{1}}\right)$ has \begin{itemize}
		\item objects $\left(X, \phi_{X}, u_{X}\right)$ where the pair $\left(X, \phi_{X}\right)$ is an object of $\mathbf{Ins}_{1}$ and $u_{X}: \phi_{X}.\eta_{X} \Rightarrow 1_{X}$ is an invertible $2$-cell in $\mathcal{A}$,
		\item arrows $\left(f, \bar{f}\right): \left(X, \phi_{X}, u_{X}\right)\rightarrow \left(Y, \phi_{Y}, u_{Y}\right)$ are $1$-cells $\left(X, \phi_{X}\right) \rightarrow \left(Y, \phi_{Y}\right)$ in $\mathbf{Ins}_{1}$ which moreover satisfy the unit coherence for $1$-cells in $\mathcal{A}^S$,
		\item $2$-cells just as in $\mathbf{Ins}_{1}$.
	\end{itemize}
	Next, form the iso-$3$-cell inserter $\mathbf{Ins}_{3}\left(\mu p_{1}p_{2}.\phi p_{2}, S\phi p_{2}.\phi p_{2}\right)$, whose universal invertible modification is the $m$ depicted below. Observe that the $2$-category $\mathbf{Ins}_{3}$ has 
	\begin{itemize}
		\item objects $\left(X, \phi_{X}, u_{X}, m_{X}\right)$ where the triple $\left(X, \phi_{X}, u_{X}\right)$ is an object of $\mathbf{Ins}_{2}$ and $m_{X}: \phi_{X}.S\phi_{X} \Rightarrow \phi_{X}.\mu_{X}$ is an invertible $2$-cell in $\mathcal{A}$,
		\item arrows $\left(f, \bar{f}\right): \left(X, \phi_{X}, u_{X}, m_{X}\right) \rightarrow \left(Y, \phi_{Y}, u_{Y}, m_{Y}\right)$ are $1$-cells $\left(X, \phi_{X}, u_{X}\right) \rightarrow \left(Y, \phi_{Y}, u_{Y}\right)$ in $\mathbf{Ins}_{2}$ which moreover satisfy the associativity coherence for $1$-cells in $\mathcal{A}^S$,
		\item $2$-cells just as in $\mathbf{Ins}_{2}$, hence also just as in $\mathbf{Ins}_{1}$. 
	\end{itemize}
	$$\begin{tikzcd}[column sep=15, font=\fontsize{9}{6}]
		&&&&
		\mathcal{A}
		\arrow[dddd,Rightarrow,shorten=30,"\phi"]
		\arrow[rrdd,"S"'{name=A}]
		\arrow[rr,"S"]
		&&
		\mathcal{A}
		\arrow[dd,"S"]
		\arrow[to=A,Rightarrow,shorten=8,"\mu"]
		&&&&&&
		\mathcal{A}
		\arrow[rr,"S"]
		\arrow[dd,Rightarrow,shorten=12,"\phi"]
		&&
		\mathcal{A}
		\arrow[dddd,"S"]
		\\
		\\
		{\mathbf{Ins}_2}
		\arrow[rr,"p_2"]
		&&
		{\mathbf{Ins}_1}\left(S{,}{1_\mathcal{A}}\right)
		\arrow[rruu,"p_1"]
		\arrow[rrdd,"p_1"']
		&&&&
		\mathcal{A}
		&&
		{\mathbf{Ins}}_2
		\arrow[ll,Rightarrow,"m"']
		\arrow[ll, no head]
		\arrow[rr,"p_2"]
		&&
		{\mathbf{Ins}_1}\left(S{,}{1_\mathcal{A}}\right)
		\arrow[rruu,"p_1"]
		\arrow[rr,"p_1"]
		\arrow[rrdd,"p_1"']
		&&
		\mathcal{A}
		\arrow[rruu,"1_\mathcal{A}"']
		\arrow[dd,Rightarrow,shorten=12,"\phi"]
		\\
		\\
		&&&&
		\mathcal{A}
		\arrow[rruu,"1_\mathcal{A}"']
		&&&&&&&&
		\mathcal{A}
		\arrow[rr,"1_\mathcal{A}"']
		&&
		\mathcal{A}
	\end{tikzcd}$$
	
	\noindent Finally, we need to form $3$-cell equifiers to specify each of the two non-derivable axioms for a pseudoalgebra. For the right unit axiom, we form the $3$-cell equifier of the identity with the $3$-cell depicted below. This results in a $2$-category $\mathbf{Equif}_{1}$ which is just like $\mathbf{Ins}_{3}$ except that the data $\left(X, \phi_{X}, u_{X}, m_{X}\right)$ now also satisfies the right unit axiom for a pseudoalgebra.
	
	$$\begin{tikzcd}
		\phi.1_\mathcal{A} \arrow[r, "1_{\phi}.Su"] & \phi.S\phi.S\eta \arrow[r, "m.1_{S\eta}"] & \phi.\mu.S\eta \arrow[r, "1_{\phi}.\rho"] & \phi.1_\mathcal{A}
	\end{tikzcd}$$
	
	\noindent For the associativity axiom, form the equifier of the following parallel pair of $3$-cells in $\mathbf{Gray}$. Observe that the resulting $2$-category will be isomorphic to $\mathcal{A}^S$. Note that we have omitted projections for clarity.
	
	$$\begin{tikzcd}[font=\fontsize{9}{6}]
		&\phi.\mu.S^{2}\phi
		\arrow[r, "m.1_{S\mu}"]
		&\phi.S\phi.\mu_{S}
		\arrow[rd, "1_{\phi}.\alpha"]
		\\
		\phi.S\phi.S^{2}\phi
		\arrow[rd, "m.1_{S^{2}\phi}"']
		\arrow[ru, "1_{\phi}.Sm"]
		&&&\phi.\mu.\mu_{S}
		\\
		&\phi.S\phi.S\mu
		\arrow[r, "1_{\phi}.\mu_{\phi}"']
		&\phi.\mu.S\mu
		\arrow[ru, "m.1_{\mu_{S}}"']
	\end{tikzcd}$$
\end{proof}

\begin{corollary}\label{TriKL weight is cofibrant}
	The weight $W$ for $\mathbf{Gray}$-enriched Kleisli objects is a connected, cellular cofibrant object in the projective model structure on $[\mathbf{Psmnd}^\text{op}, \mathbf{Gray}]$.
\end{corollary}

\begin{proof}
	In the proof of Proposition \ref{Pseudoalgebras via cofibrant weighted limits} we described how the Eilenberg-Moore object can be constructed in $\mathbf{Gray}$ using pullbacks of powers of generating cofibrations. One implements the dual construction in $[\mathbf{Psmnd}^\text{op}, \mathbf{Gray}]$ to build $W$. In particular, start with the pseudomonad on the representable $\mathbf{Psmnd}^\text{op}\rightarrow \mathbf{Gray}$ determined by the Yoneda embedding $\mathbf{Psmnd}\rightarrow [\mathbf{Psmnd}^\text{op}, \mathbf{Gray}]$, in place of the pseudomonad $\left(\mathcal{A}, S\right)$ in $\mathbf{Gray}$. Trace through the same construction replacing every pullback with the corresponding pushout, and every instance of a power by $i_{n}, n \in \{2, 3, 4\}$ (or power by $k$) with a copower by $i_{n}$ (resp. or copower by $k$).
\end{proof}

\begin{remark}\label{weight for lax algebras and pseudo maps is projective cofibrant}
	We observe that the forgetful $2$-functor from the $2$-category of \emph{lax} algebras, pseudomorphisms and $2$-cells is also possible to construct as a composite of pullbacks of powers of cofibrations. One replaces every instance of $k$ with $i_{3}$, and performs an additional $3$-cell equifier for the left unit coherence described in Remark \ref{Remark derivable law for pseudoalgebras}, as this is no longer derivable in the lax case. Performing the dual construction starting with $\mathbf{Psmnd} \rightarrow [\mathbf{Psmnd}^\text{op}, \mathbf{Gray}]$ in $[\mathbf{Psmnd}^\text{op}, \mathbf{Gray}]$ one finds that the corresponding weight $W'$ is also cofibrant. The weighted colimit of $\left(\mathcal{A}, S\right)$ weighted by this new weight $W'$ has a generators and relations presentation similar to $\mathcal{A}_{S}$, except that the $2$-cells $u_{X}$ and $m_{X}$ are no longer required to be invertible and the right unit coherence needs to be specified as it is no longer derivable. The limit variant of $W'$ produces the $2$-category of \emph{lax algebras}, pseudomorphisms and $2$-cells.
\end{remark}

\section{Conclusion}

\noindent In Subsection \ref{Section definition and presentation of Gray enriched Kleisli objects for pseudomonads} we defined a notion of Kleisli object of a pseudomonad in terms of a $\mathbf{Gray}$-enriched weighted colimit. Applying the enriched colimit construction corresponding to this weight to a pseudomonad $\left(\mathcal{A}, S\right)$ in $\mathbf{Gray}$ yields the $2$-category $\mathcal{A}_{S}$ which is neither the Kleisli bicategory nor the $2$-category of free pseudoalgebras. Rather, $\mathcal{A}_{S}$ can be described via a presentation in terms of generators and relations, in which each object $X \in \mathcal{A}$ has been equipped with a pseudoalgebra structure, each map has been equipped with a pseudomorphism structure, and each $2$-cell forms a $2$-cell of pseudoalgebras (Subsection \ref{Subsection the presentation of the Kleisli object of a pseudomonad}). The canonical comparison $2$-functor from this $\mathbf{Gray}$-enriched Kleisli object $\mathcal{A}_{S}$ to any other pseudoadjunction $L \dashv R$ splitting the original pseudomonad has been shown to be bi-fully faithful in Theorem \ref{comparison from Gray Kleisli is bi-fully-faithful}. In particular, the comparison to the Eilenberg-Moore object to pseudoalgebras also allows for a simpler description of the $2$-cells in $\mathcal{A}_{S}$ in terms of $2$-cells of pseudoalgebras. This is described in Corollary \ref{2-cells in A_S are 2-cells of free pseudoalgebras}.
\\
\\
\noindent As per Theorem \ref{when is comparison from gray kleisli a biequivalence}, the comparison is a weak biequivalence if and only if the left pseudoadjoint $L$ is biessentially surjective on objects. Moreover, this is also logically equivalent to an `up to biequivalence' version of the universal property for $\mathbf{Gray}$-enriched Kleisli objects for pseudomonads. The property of being biessentially surjective on objects is closed under composition of $2$-functors, while in contrast $\mathbf{Gray}$-enriched Kleisli left pseudoadjoints are not closed under composition as recorded in Corollary \ref{Gray Kleisli not closed under composition}. Closure under composition of Kleisli left adjoints is an essential ingredient in the explicit description of free cocompletions of $2$-categories under Kleisli objects of monads, as developed in \cite{FTM2}. With these goals in mind, Section \ref{Section Homotopy Theoretic Aspects} established nice homotopical properties for this weight, such as projective cellular cofibrancy. As we will discuss in the forthcoming \cite{Miranda Tricategorical Limits and Colimits}, projective cofibrancy of the weight means that $\mathbf{Gray}$-enriched Kleisli objects also enjoy a tricategorical universal property. 

\section{Appendix}

\subsection{Explicit biequivalence inverse to the comparison}\label{Appendix explicit inverse to comparison}

\noindent Let $K: \mathcal{A}_{S} \rightarrow \mathcal{B}$ be the comparison $2$-functor of Theorem \ref{comparison from Gray Kleisli is bi-fully-faithful}, consider the bijective on objects- $2$-fully faithful factorisation of $K$ and denote its bijective on objects part by $C: \mathcal{A}_S \rightarrow \mathcal{B}'$. By Theorem \ref{comparison from Gray Kleisli is bi-fully-faithful}, the $2$-functor $C$ is a weak biequivalence and hence underlies a biadjoint biequivalence in $\mathbf{Bicat}$. Proposition \ref{D pseudofunctor pseudoinverse to C} describes this biadjoint biequivalence explicitly. This involves giving data as listed below.

\begin{itemize}
	\item A pseudofunctor $D: \mathcal{B}' \rightarrow \mathcal{A}_S$, described in part (1),
	\item Two invertible icons $\gamma: DC \rightarrow 1_{\mathcal{A}_{S}}$ and $\beta: 1_{\mathcal{B}'} \rightarrow CD$, described in parts (2) and (3) respectively,
	\item An invertible modification $u: D\left(\beta\right).\gamma_{D} \Rrightarrow 1_{D}$, described in part (4).
\end{itemize}

\begin{proposition}\label{D pseudofunctor pseudoinverse to C}
	\hspace{1mm}
	\begin{enumerate}
		\item There is an identity on objects pseudofunctor $D: \mathcal{B}' \rightarrow \mathcal{A}_S$ which
		
		\begin{itemize}
			\item sends $p:\bar{F}X \rightarrow \bar{F}Y$ to $\varepsilon_{Y}.Up.\eta_{X}$,
			\item sends $\phi: p \Rightarrow q$ to $\varepsilon_{Y}.\phi.\eta_{X}$,
			\item has unitor $D_{X}: 1_{DX} \rightarrow D\left(1_{X}\right)$ given by $u_{X}$
			\item has compositor $D_{q, p}: DqDp \Rightarrow D\big(q.p\big)$ on \begin{tikzcd}X \arrow[r, "p"] &Y \arrow[r, "q"] & Z\end{tikzcd} given by the pasting depicted below, in which we note that $\mu = U{\bar{\varepsilon}}_{\bar{F}}$.
		\end{itemize}		
		$$\begin{tikzcd}[font=\fontsize{9}{6}]
			X \arrow[rr,"{\eta}_{X}"]
			&& SX \arrow[rr,"Up"]
			&& SY \arrow[rr,"{\varepsilon}_{Y}"]
			\arrow[rrdd,"S{\eta}_{Y}"] \arrow[dddd,"{\text{id}}_{SY}"'{name=A}, bend right]
			&& UY \arrow[rrdd,"{\eta}_{Y}"] \arrow[dd,"{\varepsilon}_{{\eta}_{Y}}",Rightarrow,shorten=3mm]
			\\
			\\
			&&&&&& {S^2}Y \arrow[rr,"{\varepsilon}_{SY}"] \arrow[dd,"SUq"'{name=C}] \arrow[ddll,"{\mu}_{Y}"'{name=B}]
			&& SY \arrow[dd,"Uq"] \arrow[lldd,"{\varepsilon}_{Uq}",Rightarrow,shorten=5mm]
			\\
			\\
			&&&& SY \arrow[rrdd,"Uq"'{name=D}]
			&& {S^2}Z \arrow[rr,"{\varepsilon}_{SZ}"] \arrow[dd,"{\mu}_{Z}"']
			&& SZ \arrow[dd,"{\varepsilon}_{Z}"] \arrow[lldd,Rightarrow,shorten=5mm,"m_{Z}"]
			\\
			\\
			&&&&&& SZ \arrow[rr,"{\varepsilon}_{Z}"']
			&& Z \arrow[from=B,to=A,"{\rho}_{Y}",Rightarrow,shorten=0.5mm,shift right=12] \arrow[from=C,to=D,Rightarrow,"U\bar{\varepsilon}_{q}"',shorten=5mm, near start]
		\end{tikzcd}$$
		\item There is an invertible icon $\gamma: D\circ C \rightarrow 1_{\mathcal{A}_{S}}$ with components on \begin{itemize}
			\item Generating maps $f: X \rightarrow Y$ from $\mathcal{A}\left(X, Y\right)$ given by either of the following equal $2$-cells.
			$$\begin{tikzcd}[font=\fontsize{9}{6}]
				X \arrow[rr,"{\eta}_X"] \arrow[rrdd,bend right,"{\text{id}}_X"'{name=A}]
				&&
				SX \arrow[to=A,Rightarrow,"u_X",shorten=3mm] \arrow[dd,"{\varepsilon}_X"] \arrow[rr,"Sf"]
				&&
				SY \arrow[lldd,Rightarrow,shorten=3mm,"{\epsilon}_f"] \arrow[dd,"{\varepsilon}_Y"]
				\\
				\\
				&& X \arrow[rr,"f"']
				&& Y
			\end{tikzcd} = \begin{tikzcd}[font=\fontsize{9}{6}]
				SX \arrow[rr,"Sf"] \arrow[rrdd,Rightarrow,"{\eta}_f",shorten=5mm]
				&& SY \arrow[rr,"{\varepsilon}_Y"{name=A}]
				&& Y
				\\
				\\
				X \arrow[uu,"{\eta}_X"]
				\arrow[rr,"f"']
				&& Y \arrow[rruu,bend right,"{\text{id}}_Y"'{name=B}] \arrow[uu,"{\eta}_Y"']
				\arrow[from=A,to=B,Rightarrow,"u_Y",shorten=2mm,shift right=2]
			\end{tikzcd}$$
			\item Generating maps $\varepsilon_{X}: SX \rightarrow X$ given by either of the following equal $2$-cells.
			
			$$\begin{tikzcd}[font=\fontsize{9}{6}]
				SX \arrow[r,"{\eta}_{SX}"] \arrow[rr,bend right=60,"{\text{id}}_{SX}"'{name=A}]
				& {S^2}X \arrow[to=A,Rightarrow,"{\lambda}_X"] \arrow[r,"{\mu}_X"]
				& SX \arrow[r,"{\varepsilon}_X"]
				& X
			\end{tikzcd} = \begin{tikzcd}[font=\fontsize{9}{6}]
				SX \arrow[rr,"{\eta}_{SX}"] \arrow[rrdd,bend right,"{\text{id}}_{SX}"'{name=A}]
				&&
				{S^2}X \arrow[to=A,Rightarrow,"u_{SX}",shorten=3mm] \arrow[dd,"{\varepsilon}_{SX}"] \arrow[rr,"{\mu}_X"]
				&&
				SX \arrow[lldd,Rightarrow,shorten=3mm,"m_X"] \arrow[dd,"{\varepsilon}_X"]
				\\
				\\
				&& SX \arrow[rr,"{\varepsilon}_X"']
				&& X
			\end{tikzcd}$$
			
			\item An arbitrary composite $gf$ in $\mathcal{A}_S$ given by \begin{tikzcd} DC\left(gf\right) \arrow[rr, "D_{Cg {,}Cf}"] && DCg.DCf \arrow[rr, "\gamma_{g}.\gamma_{f}"] && gf \end{tikzcd}
		\end{itemize} 
		\item There is an invertible icon $\beta: 1_{\mathcal{B}'} \rightarrow C\circ D$ with component at $p: X \rightarrow Y$ given by the $2$-cell determined by the following pasting in $\mathcal{B}'$.
		
		$$\begin{tikzcd}[font=\fontsize{9}{6}]
			\bar{F}X \arrow[rr,"\bar{F}{\eta}_X"{name=A}] \arrow[rrdd,"{\text{id}}_{\bar{F}X}"'{name=B},bend right]
			&& \bar{F}U\bar{F}X \arrow[rr,"\bar{F}Up"] \arrow[dd,"\bar{\varepsilon}_{\bar{F}X}"']
			&& \bar{F}U\bar{F}Y \arrow[dd,"\bar{\varepsilon}_{\bar{F}Y}"] \arrow[lldd,"\bar{\varepsilon_p}",Rightarrow,shorten=5mm]
			\\
			\\
			&& \bar{F}X \arrow[rr,"p"']
			&& \bar{F}Y \arrow[from=A,to=B,Rightarrow,shorten=3mm,"\bar{u}_X",shift left=5]
		\end{tikzcd}$$
		\item There is an invertible modification $\Omega: D\left(\beta\right).\gamma_{D} \Rrightarrow 1_{D}$ whose component on $X$ is given by $u_{X}$.
		\item There is an equality of icons $\beta_{C}.C\left(\gamma\right) = 1_{C}$
		\item The data described in parts (1)-(5) form a biequivalence internal to $\mathbf{Bicat}$.
	\end{enumerate}
\end{proposition}

\noindent We do not give a proof for Proposition \ref{D pseudofunctor pseudoinverse to C} as the explicit description of the biequivalence inverse that it gives is not needed for any of the proofs in this paper. By Theorem \ref{comparison from Gray Kleisli is bi-fully-faithful} a biequivalence inverse indeed must exist, and we have just specialised the general way in which such biequivalence inverses can be explicitly constructed.

\subsection{Proof of Lemma \ref{iso 3 cell inserter from 3 cell inserter and 3 cell equifier}}\label{Proof of Lemma 3-cell inserter}
\begin{proof}
	Observe that $k$ is the composite of the middle column in the diagram below, in which all squares are pushouts and each $i_{n}$ for $n \in \{3, 4\}$ is a generating cofibration in the Lack model structure on $2$-$\mathbf{Cat}$, recalled in Example \ref{Example Lack model structure}. Since cofibrations are stable under pushout (Corollary 1.1.11 of \cite{Hovey Model Categories}) and closed under composition, it follows that $k$ is also a cofibration. Note that data appearing in black are objects, morphisms and $2$-cells \textit{in} certain $2$-categories, while data appearing in \color{red}{red}\color{black} \hspace{1mm} are $2$-functors \textit{between} these $2$-categories.
	
	\begin{tikzcd}[column sep=12, row sep = 15, font=\fontsize{9}{6}]
		&&&&&&&&&&
		X
		\arrow[rrrr,bend left, "f"{name=A}]
		\arrow[rrrr,bend right,"g"'{name=B}]
		&&
		{} \arrow[dddddd,shorten=30,red,"i_3"]
		&&
		Y
		\\
		\\
		\\
		\\
		\\
		\\
		X
		\arrow[rrrr,bend left, "f"{name=A}]
		\arrow[rrrr,bend right,"g"'{name=B}]
		&&
		{}
		\arrow[dddddd,shorten=30,red,"i_3"']
		&&
		Y
		\arrow[rrrrrr,shorten=8,red,"i_3"]
		&&&&&&
		X
		\arrow[rrrr,bend left, "f"{name=A}]
		\arrow[rrrr,bend right,"g"'{name=B}]
		&&
		{}
		\arrow[dddddd,shorten=38,red]
		&&
		Y
		\arrow[from=A,to=B,Rightarrow,"\alpha"'{name=C},shorten=5]
		\\
		\\
		\\
		\\
		\\
		\\
		X
		\arrow[rrrr,bend left, "f"{name=A}]
		\arrow[rrrr,bend right,"g"'{name=B}]
		&&
		{}
		&&
		Y
		\arrow[rrrrrr,red,shorten=8]
		\arrow[from=A,to=B,Leftarrow,"\alpha^{-1}"'{name=D},shift left=2,shorten=5]
		&&&&&&
		X
		\arrow[rrrr,bend left=40, "f"{name=A}]
		\arrow[rrrr,bend right=40,"g"'{name=B}]
		&&
		{}
		\arrow[dddddd,shorten=34,red]
		&&
		Y
		\arrow[from=A,to=B,Rightarrow,"\alpha"'{name=C},shift right=3,shorten=5]
		\arrow[from=A,to=B,Leftarrow,"\alpha^{-1}"{name=D},shift left=2,shorten=5]
		&&&&&&
		X
		\arrow[rrrr,bend left=40, "f"{name=A}]
		\arrow[rrrr,bend right=40,"f"'{name=B}]
		\arrow[llllll,red,shorten=8]
		&&
		{}
		\arrow[dddddd,shorten=34,red,"i_4"]
		&&
		Y
		\arrow[from=A,to=B,Rightarrow,"\alpha^{-1}\alpha"'{name=C},shift right=2,shorten=5]
		\arrow[from=A,to=B,Rightarrow,"1_{f}"{name=D},shift left=5,shorten=5]
		\\
		\\
		\\
		\\
		\\
		\\
		X
		\arrow[rrrr,bend left=40, "g"{name=A}]
		\arrow[rrrr,bend right=40,"g"'{name=B}]
		&&
		{}
		\arrow[dddddd,red,shorten=34,"i_4"']
		&&
		Y
		\arrow[rrrrrr,red,shorten=8]
		\arrow[from=A,to=B,Rightarrow,"\alpha\alpha^{-1}"'{name=C},shift right=2,shorten=5]
		\arrow[from=A,to=B,Rightarrow,"1_{g}"{name=D},shift left=5,shorten=5]
		&&&&&&
		X
		\arrow[rrrr,bend left=40, "f"{name=A}]
		\arrow[rrrr,bend right=40,"g"'{name=B}]
		&&
		{}
		\arrow[dddddd,shorten=34,red]
		&&
		Y
		\arrow[from=A,to=B,Rightarrow,"\alpha"'{name=C},shift right=4,shorten=5]
		\arrow[from=A,to=B,Leftarrow,"\alpha^{-1}"{name=D},shift left=2,shorten=5]
		&&&&&&
		X
		\arrow[rrrr,bend left=40, "f"{name=A}]
		\arrow[rrrr,bend right=40,"f"'{name=B}]
		\arrow[llllll,shorten=8,red]
		&&
		=
		&&
		Y
		\arrow[from=A,to=B,Rightarrow,"\alpha^{-1}\alpha"'{name=C},shift right=2,shorten=5]
		\arrow[from=A,to=B,Rightarrow,"1_{f}"{name=D},shift left=3,shorten=5]
		\\
		\\
		\\
		\\
		\\
		\\
		X
		\arrow[rrrr,bend left=40, "g"{name=A}]
		\arrow[rrrr,bend right=40,"g"'{name=B}]
		&&
		=
		&&
		Y\arrow[rrrrrr, red]
		\arrow[from=A,to=B,Rightarrow,"\alpha\alpha^{-1}"'{name=C},shift right=3,shorten=5]
		\arrow[from=A,to=B,Rightarrow,"1_{g}"{name=D},shift left=4,shorten=5]
		&&&&&&
		X
		\arrow[rrrr,bend left=40, "f"{name=A}]
		\arrow[rrrr,bend right=40,"g"'{name=B}]
		&&
		\cong_{\alpha}
		&&
		Y
	\end{tikzcd}	
\end{proof}

\end{document}